\documentclass[a4paper,12pt]{amsart}
\usepackage{lipsum}
\usepackage{amssymb,amsthm}
\usepackage[foot]{amsaddr}
\usepackage{amsmath}
\usepackage{tikz}
\usetikzlibrary{quotes,angles}
\usepackage{graphicx}
\usepackage{subcaption}
\usepackage[shortlabels]{enumitem}
\usetikzlibrary{arrows}
\usepackage{float}
\usepackage{graphicx}
\usepackage{subcaption}
\usepackage{hyperref}
\usepackage[section]{placeins}
\graphicspath{{revolution_figs/}}
\usetikzlibrary{decorations.pathmorphing}
\tikzset{snake it/.style={decorate, decoration=snake}}
\usepackage{pgfplots}

\pgfplotsset{compat=1.18}
\makeatletter
\newcommand*{\rom}[1]{\expandafter\@slowromancap\romannumeral #1@}
\makeatother

\numberwithin{equation}{section}

\theoremstyle{plain}
\newtheorem{theorem}{Theorem}
\numberwithin{theorem}{section}

\theoremstyle{definition}
\newtheorem{definition}[theorem]{Definition}
\theoremstyle{remark}
\newtheorem{remark}[theorem]{Remark}
\newtheorem{example}[theorem]{Example}
\theoremstyle{remark}
\newtheorem{discussion}[theorem]{Discussion}
\theoremstyle{remark}

\newcommand{\smo}{\setminus \mathbf{0}}

\newcommand{\norm}[1]{\left\lVert#1\right\rVert}      % Norm
\newcommand{\abs}[1]{\left|#1\right|}                 % Absolutbetrag
\newcommand{\paren}[1]{\left(#1\right)}               % Klammern
\newcommand{\sparen}[1]{\left\{#1\right\}}      % Mengenklammer
\newcommand{\dd}{\mathrm{d}}  % without the space
\newcommand{\dalpha}{\mathbf{\alpha}}%for differentiation

\newcommand{\supp}{\operatorname{supp}} % singular support

\newcommand{\Cc}{\mathcal{C}}

\newcommand{\Dc}{\mathcal{D}}
\newcommand{\Ec}{\mathcal{E}}
\newcommand{\Fc}{\mathcal{F}}

\newcommand{\Rc}{\mathcal{R}}

\newcommand{\Sc}{\mathcal{S}}

\newcommand{\WF}{\mathrm{WF}}                         % Wavefront set
\newcommand{\wf}{\mathrm{WF}}                         % Wavefront set

\newcommand{\partyf}[2]{\frac{\partial #2}{\partial y_{#1}}}

\newcommand{\vv}{{\mathbf{v}}}
\newcommand{\vw}{{\mathbf{w}}}

\newcommand{\bpm}{\begin{pmatrix}}
\newcommand{\epm}{\end{pmatrix}}

\newcommand{\xn}{x_n}
\newcommand{\yn}{y_n}

\newcommand{\vx}{{\mathbf{x}}}
\newcommand{\vxo}{\mathbf{x}_0}

\newcommand{\vxp}{\vx''}
\newcommand{\vyp}{\vy''}

\newcommand{\vy}{{\mathbf{y}}}

\newcommand{\vs}{\mathbf{s}}

\newcommand{\vxi}{{\boldsymbol{\xi}}}
\newcommand{\vxio}{{\boldsymbol{\xi}_0}}

\newcommand{\vsig}{{\boldsymbol{\sigma}}} %when $\sigma$ is one dimensional, 
\newcommand{\rr}{{{\mathbb R}}}

\newcommand{\rtwo}{{{\mathbb R}^2}}
\newcommand{\rthree}{{{\mathbb R}^3}}
\newcommand{\rn}{{{\mathbb R}^n}}
\newcommand{\rnm}{{{\mathbb R}^{n-1}}}

\newcommand{\drn}{\dot{\mathbb{R}^n}}

\newcommand{\st}{\hskip 0.3mm : \hskip 0.3mm}

\newcommand{\be}{\begin{equation}}
\newcommand{\bea}{\begin{eqnarray}}
\newcommand{\eea}{\end{eqnarray}}
\newcommand{\bean}{\begin{eqnarray*}}
\newcommand{\eean}{\end{eqnarray*}}

\newcommand{\bel}[1]{\begin{equation}\label{#1}}
\newcommand{\ee}{\end{equation}}
\newcommand{\eel}[1]{{\label{#1}\end{equation}}}
\newcommand{\intt}{{\operatorname{int}}}

\newcommand{\bd}{{\operatorname{bd}}}

\DeclareMathOperator{\sinc}{sinc}

\newcommand{\nul}{\operatorname{null}}
\newcommand{\spann}{\operatorname{span}}

\usepackage{fullpage}

\usepackage{kbordermatrix}
\renewcommand{\kbldelim}{(}% Left delimiter
\renewcommand{\kbrdelim}{)}% Right delimiter

\usepackage{datetime}
\usetikzlibrary{quotes, angles}

\title[short]{Surface of revolution Radon transforms with centers on generalized surfaces in $\mathbb{R}^n$\\{\footnotesize\ddmmyyyydate\today~\currenttime}}
\author{James W. Webber\textsuperscript{$\dagger$}}
\author{Sean Holman\textsuperscript{$\ddagger$}}
\author{Eric Todd Quinto*}
\address[James W. Webber (corresponding author)]{Department of Oncology and Gynecology, Brigham and Women's Hospital, 221 Longwood Ave., Boston, MA 02115}
\address[Sean Holman]{Department of Mathematics, The University of Manchester, Alan Turing Building, Oxford Road, Manchester M13 9PY}
\address[Eric Todd Quinto]{Department of Mathematics, Tufts
University, 177 College Ave, Medford, MA 02155}
\email[A1,A2]{jwebber5@bwh.harvard.edu\textsuperscript{$\dagger$}, sean.holman@manchester.ac.uk\textsuperscript{$\ddagger$}}
\email[A3]{todd.quinto@tufts.edu*}

\providecommand{\keywords}[1]
{
  \small	
  \textbf{\textit{Keywords---}} #1
}

\begin{document}

\begin{abstract}
We present a novel analysis of a Radon transform, $R$, which maps an $L^2$ function of compact support to its integrals over smooth surfaces of revolution with centers on an embedded hypersurface in $\mathbb{R}^n$. Using microlocal analysis, we derive necessary and sufficient conditions relating to $R$ for the Bolker condition to hold, which has implications regarding the existence and location of image artifacts. We present a general inversion framework based on Volterra equation theory and known results on the spherical Radon transform, and we prove injectivity results for $R$. Several example applications of our theory are discussed in the context of, e.g., Compton Scatter Tomography (CST) and Ultrasound Reflection Tomography (URT). In addition, using the proposed inversion framework, we validate our microlocal theory via simulation, and present simulated image reconstructions of image phantoms with added noise.
%\\
%\\
%\tred{En Fran\c{c}ais - Nous d{\'e}voilons une nouvelle m{\'e}thode d'analyse du the{\'e}orème de Radon, $R$, qui se doit d'{\'e}riger la fonction $L^2$ pour un support compact des int{\'e}grales de surfaces de r{\'e}volution lisses centr{\'e} sur une hypersurface int{\'e}gr{\'e}e en $\mathbb{R}^n$. Afin d'assurer la validit{\'e} de l'{\'e}tat de Bolker, nous d{\'e}terminons les conditions requises et ad{\'e}quates via analyse microlocale ; entraînant des incidences sur l'existence et l'emplacement des artefacts d'image. Nous d{\'e}montrons un cadre d'inversion g{\'e}n{\'e}rique fond{\'e} sur la the{\'e}orie d'{\'e}quation Volterra et des r{\'e}sultats {\'e}tablis par la transform{\'e}e de Radon ; en outre, nous prouvons les r{\'e}sultats d'injectivit{\'e} pour $R$. Nous pr{\'e}sentons de nombreux exemples d'application de notre the{\'e}orie au sein de cadres tels que Compton Scatter Tomography (CST) [tomographie par effet Compton] ou Ultrasound Reflection Tomography (URT) [tomographie ultrasonore].
%Finalement, notre the{\'e}orie microlocale est valid{\'e}e dans le
%cadre d'inversion mentionn{\'e} par voie de simulation ; nous
%incorporons des fant\^{o}mes num{\'e}riques enrichis de bruit.
%}

\end{abstract}
\maketitle

 - surfaces of revolution, generalized Radon transforms, inversion methods, microlocal analysis

%\keywords{dimensionality reduction, orthogonal projections, miRNA expression analysis, cancer prediction}

\section{Introduction}
In this paper, we analyze a novel generalized Radon transform, $R$, which gives the integrals of a compactly supported $L^2$ function over surfaces of revolution with centers on a smooth hypersurface in $\mathbb{R}^n$. We investigate the inversion stability of $R$ using microlocal analysis, and show that $R$ is injective using known theory on the spherical Radon transform.

There is now a wealth of literature covering the inversion and stability properties of surface of revolution Radon transforms \cite{andersson1988determination,kunyansky2007explicit,rubin2002inversion,haltmeier2017spherical,Web,SU:SAR2013,Caday:SAR,Rubin-sonar, Nguyen-Pham, Terzioglu-KK-cone, GouiaAmbartsoumian2014,  WHQ, Q2006:supp, rigaud20183d, rigaud20213d, cebeiro2021three,  webber2021microlocal, webber2023cylindrical, grathwohl2020imaging}, which have applications in, e.g., CST \cite{rigaud20183d}, Emission CST (ECST) \cite{Terzioglu-KK-cone}, seismic imaging \cite{grathwohl2020imaging}, Synthetic Aperture Radar (SAR) \cite{SU:SAR2013}, and URT \cite{WHQ}. 

In \cite{webber2021microlocal}, the authors consider the microlocal
properties of a Radon transform, denoted $R$ (using the notation of
\cite{webber2021microlocal}), which defines the integrals of an $L^2$
function over the surfaces of revolution of continuous curves defined by a
function, $q$. The surfaces of revolution considered have centers on a flat
plane, and the axes of revolution are perpendicular to the plane of
centers. In this case, $R$ is shown to be an elliptic Fourier Integral
Operator (FIO) under certain conditions on $q$. Further, the authors
provide necessary and sufficient conditions on $q$ for $R$ to satisfy the Bolker condition, which has important implications regarding image artifacts. Simulated image reconstructions are shown to verify the microlocal theory in specific examples where the Bolker condition does, and does not hold.

In \cite{rigaud20183d}, the authors present contour reconstruction methods for FIO which define the integrals of a function over lemons, i.e., the surfaces of revolution of circular arcs. A lemon surface has two points of self-intersection. In the geometry of \cite{rigaud20183d}, one of the points of self-intersection of the lemon is fixed on the surface of a sphere, and the other moves around on the same sphere. Using microlocal analysis and filtered backprojection ideas, the authors show that the image contours can be recovered from lemon integral data. Simulated image reconstructions of image phantoms are presented with varying levels of added noise, and Total Variation (TV) denoising is applied to help combat noise.

In \cite{terzioglu2019some}, the author derives inversion formulae for a cone Radon transform, $C^k$, in $n$ dimensions, and investigates the microlocal properties of $C^k$. The cones considered have fixed central axis direction. All cones opening angles and vertex positions are used. The authors prove that $C^k$ and its dual are FIOs, and show that the normal operator of $C^k$ is pseudodifferential operator, which has important implication regarding the stability of inversion of $C^k$. The range of $C^k$ is also explored, and the authors present a differential equation which $C^k$ satisfies, which takes steps towards characterizing the range of $C^k$. The same author also provides explicit reconstruction formulae for cone transforms with cone points in fairly general positions in \cite{terzioglu2023recovering}.

In \cite{kunyansky2007explicit}, the authors present explicit inversion
formulae for a Radon transform which is defined by the integrals of an $n$-dimensional function over ($n-1$)-dimensional spheres with centers on the boundary of a closed ball, $B$, which has radius $R$, center zero. The target function, $f$, is assumed to be supported on $B$. The derivation uses certain properties of the Helmholtz equation in $\mathbb{R}^n$, and an integral representation for $f$ which involves a convolution of $f$ with  a solution to the Helmholtz equation. This leads to an inversion formula of filtered backprojection type. Using the proposed inversion formulae, the authors present simulated reconstructions of characteristic functions in two and three-dimensions with added noise.

In this paper, we present a novel inversion framework and microlocal
analyses for a new Radon transform, $R$, which is defined by the integrals of a compactly supported $L^2$ function, $f$, over surfaces of revolution with centers on a  hypersurface in $\mathbb{R}^n$. This work has important application in, e.g., CST, ECST, and URT. The surfaces of integration we consider are the surfaces of revolution of smooth curves, which are defined by a function, $h$. The surface of centers we consider is of the form $S = Q \times \mathbb{R} \subset \mathbb{R}^n$, where $Q \subset \mathbb{R}^{n-1}$ is an $(n-2)$-dimensional embedded hypersurface. $S$ can be thought of as a generalized cylinder in the sense that in the special case when $Q = S^{n-2}$ (i.e., when $Q$ is a sphere, dimension $n-2$), $S$ defines a cylinder in $\mathbb{R}^n$. Such center surfaces have been considered previously in \cite{haltmeier2017spherical}, where the authors provide explicit inversion formulae for a spherical Radon transform. Using microlocal analysis, we provide necessary and sufficient conditions on $S$ and $h$ for $R$ to be a nondegenerate FIO which satisfies the Bolker condition. This has important implications regarding the existence of image artifacts.  Using a combination of Volterra integral equation theory \cite{Tricomi} and known results on the spherical Radon transform \cite{Q2006:supp}, we prove injectivity results for $R$, and provide a novel inversion framework to recover $f$ from $Rf$ data. Specific cases of this theory have been considered previously in, e.g., \cite{webber2023cylindrical,WHQ}. The results of \cite[sections 3 and 4]{webber2023cylindrical} are a special case of our theory when $S$ is a cylinder in $\mathbb{R}^3$ and $h$ defines a circular arc, and thus the work presented here is a direct generalization of the theory of \cite{webber2023cylindrical}. In \cite{WHQ}, the authors present microlocal analyses of ellipsoid and hyperboloid Radon transforms, and they provide conditions for the Bolker condition to be satisfied. The surfaces of revolution we consider are more general than ellipsoids and hyperboloids of revolution, and thus the microlocal theory we present here is not covered by \cite{WHQ}. 

In \cite{webber2021microlocal}, surface of revolution transforms are also considered, and the authors derive conditions which are equivalent to the Bolker condition. In \cite{webber2021microlocal}, the integral surface centers are constrained to a flat plane, and the surfaces of revolution have central axes which are perpendicular to the plane of centers. In this work, we consider more general center surfaces than a flat plane (i.e., the $S$ as described above), and the surfaces of revolution we consider have axes of revolution which are embedded in $S$. Thus, the theory of \cite{webber2021microlocal} does not apply to this work.

The surface of centers we consider  ($S$) bears similarities with that of \cite{haltmeier2017spherical}. In \cite{haltmeier2017spherical}, only spherical integral surfaces are considered. We consider much more general integrals surfaces, which are the surfaces of revolution of smooth curves defined by a function, $h$. Spherical integral surfaces are a special case of our theory, when $h$ is set to define a semicircular curve.

In addition to the microlocal and injectivity theory presented, we also discuss multiple applications of our theory, e.g., in CST, ECST, and URT, and, using the proposed inversion method, we present simulated image reconstructions in the context of CST and URT.

The remainder of this paper is organized as follows. In section \ref{sect:defns}, we state some preliminary definitions and theory from microlocal analysis that we apply later to prove our theorems. In section \ref{bolker}, we introduce a new Radon transform, $R$, and prove our first main theorem which gives conditions for $R$ to be an FIO which satisfies the Bolker condition. In section \ref{inversion} we prove injectivity results for $R$ and present an inversion framework. In section \ref{examples}, we discuss some example applications of our theory to CST, ECST, and URT. To finish, we present simulated image reconstructions in section \ref{images} to validate our microlocal theory, and we present reconstructions of image phantoms with added noise.

\section{Definitions}\label{sect:defns} 

In this section, we review some theory from microlocal analysis which will be used in our theorems. We first provide some
notation and definitions.  Let $X$ and $Y$ be open subsets of
{$\mathbb{R}^{n_X}$ and $\mathbb{R}^{n_Y}$, respectively.}  Let $\Dc(X)$ be the space of smooth functions compactly
supported on $X$ with the standard topology and let $\mathcal{D}'(X)$
denote its dual space, the vector space of distributions on $X$.  Let
$\Ec(X)$ be the space of all smooth functions on $X$ with the standard
topology and let $\mathcal{E}'(X)$ denote its dual space, the vector
space of distributions with compact support contained in $X$. Finally,
let $\Sc(\rn)$ be the space of Schwartz functions, that are rapidly
decreasing at $\infty$ along with all derivatives. See \cite{Rudin:FA}
for more information. 

If $A\subset \rn$, then $\intt(A)$ and $\bd(A)$ are, respectively, the interior of $A$ and the boundary of $A$.

We now list some notation conventions that will be used throughout this paper:
\begin{enumerate}
\item For a function $f$ in the Schwartz space $\Sc(\mathbb{R}^{n_X})$ or in
$L^2(\mathbb{R}^{n_X})$, we use $\mathcal{F}f$ and $\mathcal{F}^{-1}f$ to denote
the Fourier transform and inverse Fourier transform of $f$,
respectively (see \cite[Definition 7.1.1]{hormanderI}). %\tred{$\mathcal{F}f$ and $\mathcal{F}^{-1}f$ are defined in terms of angular frequency.}
%Note that
%$$\Fc\inv \Fc f(\vx)= \frac{1}{(2\pi)^{n_X}}\int_{\vy\in\mathbb{R}^{n_X}}\int_{\vz\in
%\mathbb{R}^{n_X}} \exp((\vx-\vz)\cdot \vy)\,
%f(\vz)\d \vz\d \vy.$$

\item We use the standard multi-index notation: if
$\alpha=(\alpha_1,\alpha_2,\dots,\alpha_n)\in \sparen{0,1,2,\dots}^{n_X}$
is a multi-index and $f$ is a function on $\mathbb{R}^{n_X}$, then
\[\partial^\alpha f=\paren{\frac{\partial}{\partial
x_1}}^{\alpha_1}\paren{\frac{\partial}{\partial
x_2}}^{\alpha_2}\cdots\paren{\frac{\partial}{\partial x_{n_X}}}^{\alpha_{n_X}}
f.\] If $f$ is a function of $(\vy,\vx,\vsig)$ then $\partial^\alpha_\vy
f$ and $\partial^\alpha_\vsig f$ are defined similarly.

\item We identify the cotangent
spaces of Euclidean spaces with the underlying Euclidean spaces. For example, the cotangent space, 
$T^*(X)$, of $X$ is identified with $X\times \mathbb{R}^{n_X}$. If $\Phi$ is a function of $(\vy,\vx,\vsig)\in Y\times X\times \rr^N$,
then we define $\dd_{\vy} \Phi = \paren{\partyf{1}{\Phi},
\partyf{2}{\Phi}, \cdots, \partyf{{n_X}}{\Phi} }$, and $\dd_\vx\Phi$ and $
\dd_{\vsig} \Phi $ are defined similarly. Identifying the cotangent space with the Euclidean space as mentioned above, we let $\dd\Phi =
\paren{\dd_{\vy} \Phi, \dd_{\vx} \Phi,\dd_{\vsig} \Phi}$.

\item For $\Omega\subset \rr^m$, we define $\dot{\Omega}
= \Omega\smo$.

\end{enumerate}

\noindent The singularities of a function and the directions in which they occur
are described by the wavefront set \cite[page
16]{duistermaat1996fourier}, which we now define.
\begin{definition}
\label{WF} Let $X$ be an open subset of $\mathbb{R}^{n_X}$ and let $f$ be a
distribution in $\mathcal{D}'(X)$.  Let $(\vx_0,\vxi_0)\in X\times
\drn$.  Then $f$ is \emph{smooth at $\vx_0$ in direction $\vxio$} if
there exists a neighborhood $U$ of $\vx_0$ and $V$ of $\vxi_0$ such
that for every $\Phi\in \Dc(U)$ and $N\in\mathbb{R}$ there exists a
constant $C_N$ such that for all $\vxi\in V$,
\begin{equation}
\left|\Fc(\Phi f)(\lambda\vxi)\right|\leq C_N(1+\abs{\lambda})^{-N}.
\end{equation}
The pair $(\vx_0,\vxio)$ is in the \emph{wavefront set,} $\wf(f)$, if
$f$ is not smooth at $\vx_0$ in direction $\vxio$.
\end{definition}

Intuitively, the elements 
$(\vxo,\vxio)\in \WF(f)$ are the point-normal vector pairs at
which $f$ has singularities; $\vxo$ is the location of the 
  singularity, and  $\vxio$ is the direction in which the
  singularity occurs.
  A
 geometric example of the wavefront set is given by the characteristic function $f$ of a domain $\Omega \subset \mathbb{R}^{n_X}$ with smooth boundary, which is $1$ on $\Omega$ and $0$ off of $\Omega$. Then the wavefront set is
\[
\WF(f) = \{(\vx,t\vv) \ : t \neq 0, \ \vx \in \partial \Omega, \ \mbox{$\vv$ is orthogonal ot $\partial \Omega$ at $\vx$}\}.
\]
In other words, the wavefront set is the set of points in the boundary of $\Omega$ together with the nonzero normal vectors to the boundary. The wavefront set is an important consideration in imaging since elements of the wavefront set will correspond to sharp features of an image.

%For example, if $f$ is the
%characteristic function on the unit ball, $B_n = \{\vx\in\mathbb{R}^n : |\vx|\leq 1\}$, in $\mathbb{R}^n$, then its
%wavefront set is $\WF(f)=\{(\vx,t\vx): \vx\in S^{n-1}, t\neq 0\}$, i.e., the
%set of points on $S^{n-1}$ (i.e., the boundary of $B$) paired with the corresponding normal vectors
%to $S^{n-1}$.

%\begin{equation}

%\end{equation}
%That is, 

The wavefront set of a distribution on $X$ is normally defined as a
subset the cotangent bundle $T^*(X)$ so it is invariant under
diffeomorphisms, but we do not need this invariance, so we will
continue to identify $T^*(X) = X \times \rr^{n_X}$ and consider $\WF(f)$ as
a subset of $X\times \dot{\rr^{n_X}}$.

%Let $X$ and $Y$ be open subsets of $\rn$, $m \in\mathbb{R}$.

 \begin{definition}[{\cite[Definition 7.8.1]{hormanderI}}] \label{ellip}We define
 $S^m(Y \times X, \mathbb{R}^N)$ to be the
set of $a\in \Ec(Y\times X\times \mathbb{R}^N)$ such that for every
compact set $K\subset Y\times X$ and all multi--indices $\alpha,
\beta, \gamma$ the bound
\[
\left|\partial^{\gamma}_{\vy}\partial^{\beta}_{\vx}\partial^{\alpha}_{\vsig}a(\vy,\vx,\vsig)\right|\leq
C_{K,\alpha,\beta,\gamma}(1+\norm{\vsig})^{m-|\alpha|},\ \ \ (\vy,\vx)\in K,\
\vsig\in\mathbb{R}^N,
\]
holds for some constant $C_{K,\alpha,\beta,\gamma}>0$. 

 The elements of $S^m$ are called \emph{symbols} of order $m$.  Note
that this symbol class is  sometimes denoted $S^m_{1,0}$.  The symbol
$a\in S^m(Y \times X,\rr^N)$ is \emph{elliptic} if for each compact set
$K\subset Y\times X$, there is a $C_K>0$ and $M>0$ such that
\bel{def:elliptic} \abs{a(\vy,\vx,\vsig)}\geq C_K(1+\norm{\vsig})^m,\
\ \ (\vy,\vx)\in K,\ \norm{\vsig}\geq M.
\ee 
\end{definition}

\begin{definition}[{\cite[Definition
        21.2.15]{hormanderIII}}] \label{phasedef}
A function $\Phi=\Phi(\vy,\vx,\vsig)\in
\Ec(Y\times X\times\dot{\mathbb{R}^N})$ is a \emph{phase
function} if $\Phi(\vy,\vx,\lambda\vsig)=\lambda\Phi(\vy,\vx,\vsig)$, $\forall
\lambda>0$ and $\mathrm{d}\Phi$ is nowhere zero. The
\emph{critical set of $\Phi$} is
\[\Sigma_\Phi=\{(\vy,\vx,\vsig)\in Y\times X\times\dot{\mathbb{R}^N}
: \dd_{\vsig}\Phi=0\}.\] 
 A phase function is
\emph{clean} if the critical set $\Sigma_\Phi = \{ (\vy,\vx,\vsig) \ : \
\mathrm{d}_\vsig \Phi(\vy,\vx,\vsig) = 0 \}$ is a smooth manifold {with tangent space defined {by} the kernel of $\mathrm{d}\,(\mathrm{d}_\sigma\Phi)$ on $\Sigma_\Phi$. Here, the derivative $\mathrm{d}$ is applied component-wise to the vector-valued function $\mathrm{d}_\sigma\Phi$. So, $\mathrm{d}\,(\mathrm{d}_\sigma\Phi)$ is treated as a Jacobian matrix of dimensions $N\times (2n+N)$.}
\end{definition}
\noindent By the {Constant Rank Theorem} the requirement for a phase
function to be clean is satisfied if
$\mathrm{d}\paren{\mathrm{d}_\vsig
\Phi}$ has constant rank.

\begin{definition}[{\cite[Definition 21.2.15]{hormanderIII} and
      \cite[section 25.2]{hormander}}]\label{def:canon} Let $X$ and
$Y$ be open subsets of $\rn$. Let $\Phi\in \Ec\paren{Y \times X \times
{\rr}^N}$ be a clean phase function.  In addition, we assume that
$\Phi$ is \emph{nondegenerate} in the following sense:
\[\text{$\dd_{\vy}\Phi$ and $\dd_{\vx}\Phi$ are never zero on
$\Sigma_{\Phi}$.}\]
  The
\emph{canonical relation parametrized by $\Phi$} is defined as
\begin{equation}\label{def:Cgenl} \begin{aligned} \Cc=&\sparen{
\paren{\paren{\vy,\dd_{\vy}\Phi(\vy,\vx,\vsig)};\paren{\vx,-\dd_{\vx}\Phi(\vy,\vx,\vsig)}}:(\vy,\vx,\vsig)\in
\Sigma_{\Phi}},
% &\hspace{1.5cm} \vs\in \rr^N\smo,   
\end{aligned}
\end{equation}
\end{definition}

\begin{definition}\label{FIOdef}
Let $X$ and $Y$ be open subsets of {$\mathbb{R}^{n_X}$ and $\mathbb{R}^{n_Y}$, respectively.} {Let an operator $A :
\Dc(X)\to \mathcal{D}'(Y)$ be defined by the distribution kernel
$K_A\in \mathcal{D}'(Y\times X)$, in the sense that
$Af(\vy)=\int_{X}K_A(\vy,\vx)f(\vx)\mathrm{d}\vx$. Then we call $K_A$
the \emph{Schwartz kernel} of $A$}. A \emph{Fourier
integral operator (FIO)} of order $m + N/2 - (n_X+n_Y)/4$ is an operator
$A:\Dc(X)\to \mathcal{D}'(Y)$ with Schwartz kernel given by an
oscillatory integral of the form
\begin{equation} \label{oscint}
K_A(\vy,\vx)=\int_{\mathbb{R}^N}
e^{i\Phi(\vy,\vx,\vsig)}a(\vy,\vx,\vsig) \mathrm{d}\vsig,
\end{equation}
where $\Phi$ is a clean nondegenerate phase function and $a$ is a
symbol in $S^m(Y \times X , \mathbb{R}^N)$. The \emph{canonical
relation of $A$} is the canonical relation of $\Phi$ defined in
\eqref{def:Cgenl}. $A$ is called an \emph{elliptic} FIO if its symbol is elliptic. An FIO is called a \emph{pseudodifferential operator} if its canonical relation $\Cc$ is contained in the diagonal, i.e.,
$\Cc \subset \Delta := \{ (\vx,\vxi;\vx,\vxi)\}$.
\end{definition}

Formula \eqref{oscint} given in Definition \ref{FIOdef} can be extended to operators which are only locally represented by kernels in the form \eqref{oscint} as in \cite{hormander}, although we will not require this in the present manuscript. We also note that pseudodifferential operators can always be defined by \eqref{oscint} with phase function $\Phi(\vy,\vx,\vsig) = (\vy-\vx)\cdot \vsig$.

%\vspace{0.4cm}

Let $X$ and $Y$ be
sets and let $\Omega_1\subset X$ and $\Omega_2\subset Y\times X$. The composition $\Omega_2\circ \Omega_1$ and transpose $\Omega_2^t$ of $\Omega_2$ are defined
\[\begin{aligned}\Omega_2\circ \Omega_1 &= \sparen{\vy\in Y\st \exists \vx\in \Omega_1,\
(\vy,\vx)\in \Omega_2}\\
\Omega_2^t &= \sparen{(\vx,\vy)\st (\vy,\vx)\in \Omega_2}.\end{aligned}\]
We now state the H\"ormander-Sato Lemma \cite[Theorem 8.2.13]{hormanderI},
which provides  the relationship between the
wavefront set of distributions and their images under FIO.

\begin{theorem}[H\"ormander-Sato Lemma]\label{thm:HS} Let $f\in \Ec'(X)$ and
let ${A}:\Ec'(X)\to \Dc'(Y)$ be an FIO with canonical relation $\Cc$.
Then, $\wf({A}f)\subset \Cc\circ \wf(f)$.\end{theorem}

Let $A$ be an FIO, then its formal adjoint $A^*$ is also an FIO, and if $\Cc$ is the canonical relation of $A$, then the
canonical relation of $A^*$ is $\Cc^t$ \cite{Ho1971}. Many imaging techniques are based
on application of the adjoint operator $A^*$ and so to understand artifacts
we consider $A^* A$ (or, if $A$ does not map to $\Ec'(Y)$, then
$A^* \psi A$ for an appropriate cutoff $\psi$). Because of Theorem
\ref{thm:HS},
\begin{equation}
\label{compo}
\wf(A^* \psi A f) \subset \Cc^t \circ \Cc \circ \wf(f).
\end{equation}
The next two definitions provide tools to analyze the composition in equation \eqref{compo}.
\begin{definition}
\label{defproj} Let $\Cc\subset T^*(Y\times X)$ be the canonical
relation associated to the FIO ${A}:\mathcal{E}'(X)\to
\mathcal{D}'(Y)$. We let $\Pi_L$ and $\Pi_R$ denote the natural left-
and right-projections of $\Cc$, projecting onto the appropriate
coordinates: $\Pi_L:\Cc\to T^*(Y)$ and $\Pi_R : \Cc\to T^*(X)$.
\end{definition}

Because $\Phi$ is nondegenerate, the projections do not map to the
zero section.  
% 
% We have the following result from \cite{hormander}.
% \begin{proposition}
% \label{prop1}
% Let $\dim(X)=\dim(Y)$. Then at any point in $\Cc$:
% \begin{enumerate}[(i)]
% \item if one of $\Pi_L$ or $\Pi_R$ is a local diffeomorphism, then the
% other map is a local diffeomorphism (so $\Cc$ is a local canonical
% graph); 
% 
% \item if one of the projections $\Pi_R$ or $\Pi_L$ is singular at a
% point in $\Cc$, then so is the other. The type of the singularity may
% be different but both projections drop rank on the same set
% \begin{equation}
% \Sigma=\{(\vy,\eta; \vx,\vsig)\in \Cc :
% \det(\mathrm{d}\Pi_L)=0\}=\{(\vy,\eta; \vx,\vsig)\in \Cc : \det
% (\mathrm{d}\Pi_R)=0\}.
% \end{equation}
% \end{enumerate}
% \end{proposition}
If $A$ satisfies our next definition, then $A^* A$ (or $A^* \psi A$) is a pseudodifferential operator
\cite{GS1977, quinto}.

\begin{definition}[Bolker condition]\label{def:bolker} Let
${A}:\Ec'(X)\to \Dc'(Y)$ be a FIO with canonical relation $\Cc$ then
{$A$} (or $\Cc$) satisfies the \emph{Bolker Condition} if
the natural projection $\Pi_L:\Cc\to T^*(Y)$ is an embedding
(injective immersion).\end{definition}

Thus, using \eqref{compo}, we see that under the Bolker condition the wavefront set of $A^* \varphi Af$ will be contained in the wavefront set of $f$. Intuitively, the reconstructed image ($A^* \varphi Af$) will only include singularities at the same positions and in the same directions as the original image ($f$). 

\section{Surface of revolution transforms} 
\label{bolker}
We will consider generalized
Radon transforms in Euclidean space of three or more dimensions that
integrate over fairly arbitrary collections of surfaces of
revolution with axes on a smooth cylindrical surface.
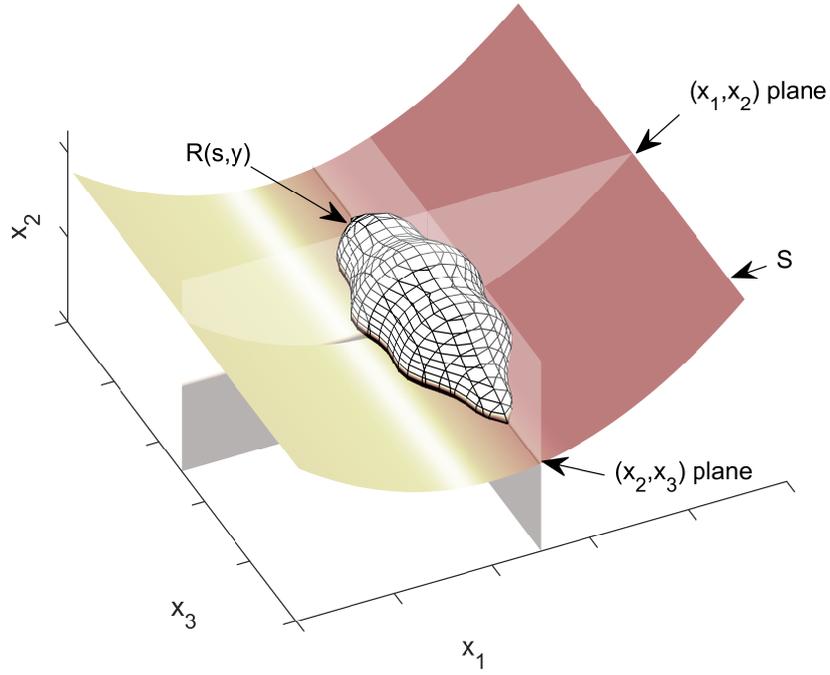
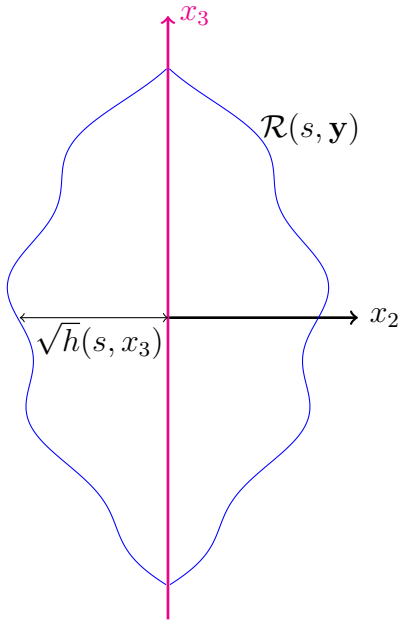
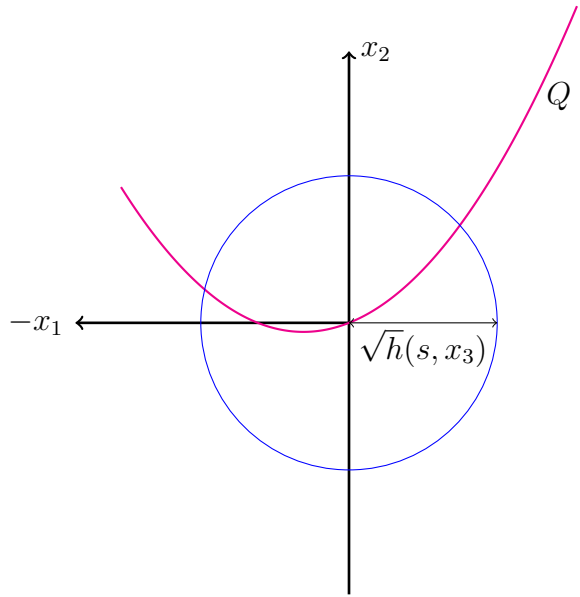
\begin{figure}[!h]
\centering
\begin{subfigure}{1\textwidth}
\centering
\includegraphics[width=0.9\linewidth, height=9cm, keepaspectratio]{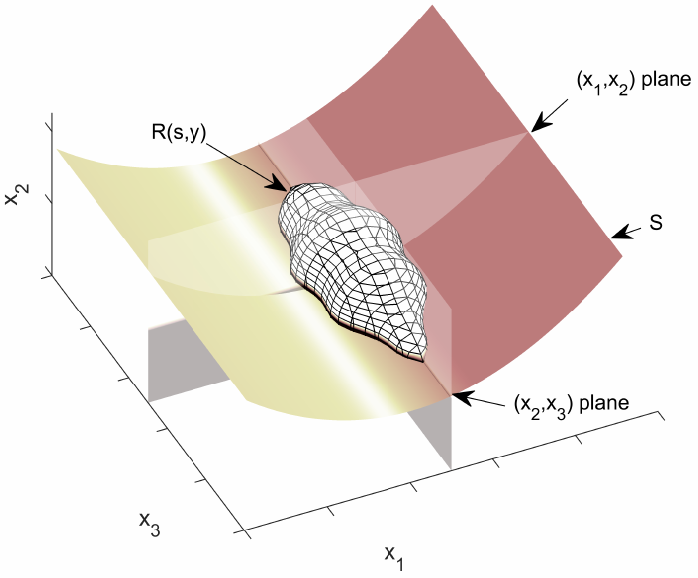}
\subcaption{3-D view} \label{1A}
\end{subfigure}
\\
\hspace{-2cm}
\begin{subfigure}{0.3\textwidth}
\centering
\begin{tikzpicture}[scale=0.5]
%\draw[fill=green,rounded corners=1mm] (0-1,0-1) \irregularcircle{1.9cm}{1mm};
\draw [->,line width=1pt] (0,0)--(5,0)node[right] {$x_2$};
\draw [->,magenta,line width=1pt] (0,-8)--(0,8)node[right] {$x_3$};
\coordinate (O) at (0,0);
\draw [samples=200,blue,domain=-35.4/5:33/5] plot ({0.33*( (-(\x*5)*(\x*5) + 100*sin((\x*5/3)*(180/3.14)) + 1200)/100 )},{\x});
\draw [samples=200,blue,domain=-35.4/5:33/5] plot ({-0.33*( (-(\x*5)*(\x*5) + 100*sin((\x*5/3)*(180/3.14)) + 1200)/100 )},{\x});
\node at (3.75,5) {$\mathcal{R}(s,\vy)$};
\draw [<->] (-3.9,0)--(0,0);
\node at (-1.8,-0.6) {$\sqrt{h}(s,x_3)$};
\end{tikzpicture}
\subcaption{$(x_2,x_3)$ plane} \label{1B}
\end{subfigure}
\hspace{1cm}
\begin{subfigure}{0.3\textwidth}
\centering
\begin{tikzpicture}[scale=0.6]
%\draw[fill=green,rounded corners=1mm] (0-1,2.6458) \irregularcircle{1.5cm}{1mm};
\draw [<-,line width=1pt] (-6,0)node[left] {$-x_1$}--(0,0);
\draw [->,line width=1pt] (0,-6)--(0,6)node[right] {$x_2$};
\draw [samples=200,thick,magenta,domain=-5:5] plot ({\x},{(\x*\x + 2*\x)/5});
\draw [blue] (0,0) circle (3.25);
\draw [<->] (0,0)--(3.25,0);
\node at (1.625,-0.55) {$\sqrt{h}(s,x_3)$};
\node at (4.6,5) {$Q$};
\end{tikzpicture}
\subcaption{$(x_1,x_2)$ plane} \label{1C}
\end{subfigure}
\caption{Example $S$ and $\mathcal{R}(s,\vy)$ when $n=3$. (A) - 3-D view. (B) - $(x_2,x_3)$ plane cross-section. (C) - $(x_1,x_2)$ plane cross-section. The plane slices in (B) and (C) are labeled in (A). The pink smooth surface in (A) is $S$, and the black and white meshed surface in (A) is $\mathcal{R}(s,\vy)$. The pink line in (B) is a line, $\ell_{\vy}$ ($\vy = 0$ in this example), as described in the main text. The pink curve in (C), is $Q$. The blue curves in (B) and (C) (e.g., the circle in (C)) are the intersections of $\mathcal{R}(s,\vy)$ with the $(x_2,x_3)$ and $(x_1,x_2)$ planes, respectively. The blue wavy curve in (B) is $\{(\pm\sqrt{h}(s,x_3),x_3)\}$.}
\label{fig1}
\end{figure}
For $\vw = (w_1,\ldots,w_n) \in \rn$, we define $\vw' = (w_1,\dots, w_{n-1})\in \rnm$ throughout this paper. 
The generalized cylindrical surface in $\rn$ is defined by a smooth embedded
hypersurface $Q\subset \rnm$ and with axis parallel the $x_n$ axis:
\bel{def:S} S = Q\times \rr.\ee
Each $\vy\in S$ is on a unique line $\ell_\vy$ contained in $S$ and
parallel the $x_n$ axis. Our surfaces of revolution will be determined
by specifying the distance from $\ell_\vy$ to each point on the
surface.   Let $h=h(s,x)$ be a nonnegative
smooth function from $\Omega_h$ to $\rr$, where $\Omega_h\subset\mathbb{R}^2$ is open, let $C = \{s\in\mathbb{R} : \exists x\in \mathbb{R}\ \text{s.t.}\  (s,x)\in\Omega_h\}$, and 
let \[Y = C\times S.\]
Our surfaces are defined as follows.  Let $(s,\vy)\in Y$ and define
\begin{equation}\label{def:Psi-Rc}
\begin{aligned}\Psi(s,\vy;\vx)&= \paren{\vx'-\vy'}\cdot
\paren{\vx'-\vy'}- h(s,x_n - y_n)\\
 \Rc(s,\vy) &=
\sparen{\vx\in \rn\st \Psi(s,\vy;\vx) = 0}.\end{aligned}
\end{equation}
When $(s,\vy)\in Y$,  $\Rc(s,\vy)$ is the surface of
revolution about $\ell_\vy$ of the smooth curve defined by $x_n\mapsto
\sqrt{h}(s,x_n-y_n)$.   Varying $s$ changes the
curve defining the surface of revolution. If one fixes $s\in C$ and
$\vy'\in Q$, then the surfaces $\Rc(s,(\vy',y_n))$ are translates of
each other along $\ell_\vy$ as $\yn$ varies. An example $S$, $\mathcal{R}(s,\vy)$, and $\sqrt{h}$ curve are illustrated in figure \ref{fig1}, in $n=3$ dimensions. Now we define our Radon transform.

\begin{definition}\label{def:Radon}
Let $D$ be an open set disjoint from $S$. Let $f\in \mathcal{D}(D)$
and $(s,\vy)\in Y$. Our Radon transform $Rf(s,\vy)$ is defined to be
the integral of $f$ over $\Rc(s,\vy)$ in surface area
measure:
\begin{equation}\label{def:R1}
Rf(s,\vy) = \int_{\mathbb{R}^n} |\nabla_{\vx}\Psi(s,\vy;\vx)| \delta(\Psi(s,\vy;\vx)) f(\vx) \ \dd \vx.
\end{equation}
\end{definition}

It is also possible to extend $R$ to $\mathcal{E}'(D)$ by continuity.

%Let $D$ be an open set disjoint from $S$. Let $f\in \Ec'(D)$
%and $(s,\vy)\in Y$. Our Radon transform $Rf(s,\vy)$ is defined to be
%the integral of $f$ over $\Rc(s,\vy)$ in surface area
%measure. 

In our results, we will put conditions on $h$ so that $\Rc(s,\vy)$ is a smooth
embedded manifold. We can now state the main theorem of
this section.

% \tc{I think $\Omega_{h,s}$ should be $\sparen{x\in \rr\st (s,x)\in \Omega_h}$
% so that the boundary of $\Omega_{h,s}$ is the set of endpoints. The set $\Omega_{h,s}$
% in the theorem is part of a line in $\rtwo$ so its boundary is its
% closure. Then I think we could replace the statement about $h$ by the
% function $x\mapsto h(s,x)$ goes to zero on $\bd(\Omega_{h,s})$ for every $s\in
% S$.\newline On second thought, I'm not sure why we don't just say
% $h\to 0$ at $\bd(\Omega_h)$ I can't think of an example when the more
% complicated condition about $\Omega_{h,s}$ would hold but $h\to 0$ at
% $\bd(\Omega_h)$ wouldn't. Also, the ``$\bd$'' condition is simpler to
% show and implies the first.} 
% 
% \jc{\tblue{We need $\Omega_{h,s}$ for the circular arc example in section 5.2. It's also a weaker condition so that's nice too. I'm happy to change the notation as you suggest, Todd. Would you mind doing this? If this is OK, don't forget to update section 5.2 with the same notation. We could maybe also add a note after defining $\Omega_{h,s}$ that we need this definition for practical examples of interest later, if this fits.}}
% 
% \tc{\tred{Got it!  In that example, the restriction  $p\in [0,\infty)$
% is not necessary mathematically but so we get lemons.  So, you're right that we need to
% keep $\Omega_{h,s}$.  I will change the definition of $\Omega_{h,s}$ so the boundary is the
% $\rr^1$ boundary (the endpoints), not the $\rr^2$ boundary (the closure of the set in $\rtwo$), as would
% be with the original definition.}}

\begin{theorem}\label{thm:bolker} Let $Q$ be an embedded smooth
hypersurface in $\rnm$ and let $S = Q\times \rr$. Let $D$ be an open
set in $\rn$, which is disjoint from $S$. 

%Let \bel{def:Df} \Df = D\cap \bigcup_{(s,\vy)\in Y} \Rc(s,\vy)\ee

%$D$ is disjoint from $S$ and

\begin{enumerate}
\item\label{boundary limit} For each $s\in C$, let \[\Omega_{h,s}
= \left\{x\in \rr : (s,x)\in \Omega_h\right\},\] and assume, for each
$s\in C$, that $h(s,\cdot)\to 0$ on the boundary of
$\Omega_{h,s}$. 
\item\label{tangent plane} Let $(s,\vy)\in Y$ and
$\vx$ be any point in $ D\cap \Rc(s,\vy)$. Assume that the reflection
of $\vx$ in the tangent plane, $T$, to $S$ at $\vy$ is not in $D$.

\item\label{hs cond} Assume for every $(s,x)\in \Omega_h$,
$h_s(s,x) \neq 0$. 

\item\label{hx/hs injective} Assume for each fixed $s\in C$, the
function $(h_{x} / h_s)(s,\cdot)$ is injective on $\{x \in\mathbb{R} : (s,x) \in\Omega_h\}$. 

\item\label{hx/hs neq 0} Assume
$\frac{\mathrm{d}}{\mathrm{d}{x}} \paren{ h_{x} / h_s } \neq 0$ on $\Omega_h$.
\end{enumerate}

Then $R:\Ec'(D)\to \Dc'(Y)$ is a nondegenerate FIO that satisfies the
Bolker condition.

Conversely, assume $R$ satisfies the Bolker assumption above $D$, and
$h$ satisfies \eqref{hs cond}. Then \eqref{tangent plane},
\eqref{hx/hs injective}, and \eqref{hx/hs neq 0} all hold.
\end{theorem}

%\tc{Note that we could define $H = \sparen{(s,x)\in C\times \rr\st h(s,x)>0}$ and for $s\in C$, define $H_s = \{x\in \rr\st (s,x)\in H\}$ and insist on the conditions \eqref{hs cond}-\eqref{hx/hs neq 0} for $H$ and $H_s$ rather than $C\times\rr$ and $\rr$ respectively.  Should I do that or put it as a remark?}

% \seanc{Move $h>0$ from (2) to definition.}
%  
% \tc{I think we need the limit condition \eqref{limit} so that
% we don't have to worry whether the surface of revolution has boundary
% points in $\supp(f)$ when $f$ is supported away from $S$.}

% \seanc{Regarding the final sentence, doesn't the Bolker condition presuppose that the operator is an FIO? Also, (2) was used in the proof to show the phase function is nondegenerate.}

%Assumption \eqref{limit} is needed so $\Rc(s,\vy)$ does not have
%boundary points in $D$. This will help insure $R$ is a standard Radon
%transform.

Note that assumption \eqref{boundary limit} holds if $h\to 0$ on
$\bd(\Omega_h)$, but this  more general version will be useful for the
example in section  \ref{sect:circular}.
Regarding assumption \eqref{tangent plane}, if $\vx\in
D\cap \Rc(s,\vy)$ and its reflection $\vx_m$ in a tangent plane $T$ to
$S$ at $\vy$ is also in $D$, then $\Pi_L$ will not be injective and
reconstructions using the normal operator can create artifacts at
$\vx_m$ from singularities at $\vx$. Condition \eqref{tangent plane} would be valid if $D$ is connected,
open and disjoint from every tangent plane to $S$. A common special
case occurs when $S$ is a smooth surface that is the boundary of a
convex cylinder; then  \eqref{tangent plane} holds
if $D$ is inside the convex cylinder. If $\vx\in T$, then $\vx$ is its own mirror point.  Therefore, we will interpret \eqref{tangent plane} to mean that $D\cap\Rc(s,\vy)$ is disjoint from $T$.

%\seanc{I'd recommend deleting the next paragraph, and rewording the one above.}

%We need \eqref{hs cond} so that \eqref{Pi_L_1} and \eqref{hx/hs} are
%defined. We use \eqref{hx/hs injective} to prove injectivity of
%$\Pi_L$. We use \eqref{hx/hs neq 0} to prove $\Pi_L$ is an immersion.
%We also show that \eqref{tangent plane} and \eqref{hx/hs neq 0} are
%necessary conditions for $\Pi_L$ to be an immersion. 

\begin{example}
Here are two interesting transforms in $\rthree$ that fit in our
theory. More general algebraic surfaces were considered in \cite{WHQ}.
Let $Q$ be a smooth hypersurface in $\rtwo$ and let $S= Q\times \rr$
and let $Y = (0,\infty)\times S$.  

\textit{The spherical transform:} Let $h(s,x) = s-x^2$.
Then $\Rc(s,\vy)$ is the sphere $(x_1-y_1)^2 + (x_2-y_2)^2 +(x_3-y_3)^2
= s$ for $(s,\vy)\in Y$.

\textit{The transform on hyperboloids of two sheets:}  Let  $h(s,x) =
x^2+s$.  Then 
$\Rc(s,\vy)$ is the hyperboloid of two sheets with axis of rotation
$\sparen{\vy'}\times \rr$ and equation
$(x_1-y_1)^2 +(x_2-y_2)^2 - (x_3-y_3)^2=s$  for $(s,\vy)\in Y$.   
\end{example}

\begin{proof}
Because the Bolker condition is local above $Y$, we will define $Q$
locally as a graph using coordinates. For convenience, we now define
$\vyp = (y_2,\dots,y_{n-1})\in \rr^{n-2}$. Let $\Omega$ be an open set
in $\rr^{n-2}$, the domain of the coordinates, and suppose $Q$ is the
graph of a function $q\in C^\infty(\Omega,\rr)$. That is, 
\[Q=\sparen{(q(\vyp),\vyp)\st \vyp\in \Omega},\qquad S =
Q\times\rr\] and 
\[(\vyp,\yn)\mapsto 
(q(\vy''),\vy'',y_n)\] give coordinates on $S$. Using these
coordinates, our Radon transform can be written
\begin{equation}\label{eq:Radon}
\begin{split}
Rf(s,(\vyp,\yn))&=\int_{\mathbb{R}^{n}}|\nabla_{\vx}\Psi|\delta\paren{\Psi(s,(\vyp,\yn);\vx)}f(\vx)\mathrm{d}\vx\\
&=\int_{-\infty}^{\infty}\int_{\mathbb{R}^{n}}|\nabla_{\vx}\Psi|f(\vx)e^{i\sigma
\Psi(s,(\vyp,\yn);\vx)}\mathrm{d}\vx\mathrm{d}\sigma,
\end{split}  
\end{equation}
where
 \[\Psi(s,(\vyp,\yn);\vx)= \paren{ x_1 -q(\vy'') }^2 + \paren{\vxp -
 \vyp}\cdot \paren{\vxp-\vyp} - h(s,x_n - y_n).\]
Note, the integral in \eqref{eq:Radon} is well-defined given that, by assumption, $D$ is bounded away from $S$, and $h$ goes to zero at the boundary of $\Omega_h$. Then, the phase function of $R$ is $\Phi =\sigma\Psi$. The weight, $|\nabla_{\vx}\Psi|$, is included following the theory of \cite{palamodov2012uniform}, so that the integrals are defined with respect to the surface measure on the surfaces of revolution.
 
%We don't need g, so:296
%  Let
%  \[g\paren{(x_1,\dots,x_{n}),\vyp} = 
%  \paren{x_1 -q(\vyp) }^2 +  \paren{\vxp -
%  \vyp}\cdot \paren{\vxp-\vyp}.\] 

We now explain why $\Phi$ is a nondegenerate phase function. From
\eqref{Pi_L} just below and assumption \eqref{hs cond} in our theorem, one sees
$\dd_s\Phi$ is never zero so the left projection is never zero. A
calculation shows that \[\dd_\vx \Phi = \sigma\paren{2[(x_1,\vxp) -
(q,\vyp)], -h_{x}(s,\xn-\yn)},\] and this is never zero
since $D$ is disjoint from the cylinder $S$. In addition, this  shows
$\Rc(s,(\vyp,\yn))\cap D$ is a smooth manifold, so $R$ is a standard Radon
transform.

The left projection of $R$ is
\begin{equation}
\label{Pi_L}
\Pi_L(s,(\vyp,\yn);\vx,\sigma) = \paren{ s, (\vyp,\yn), \overbrace{-\sigma
h_s}^{\mathrm{d}_s\Phi}, \overbrace{-2\sigma\left[(x_1 -q) \nabla
q +(\vxp-\vyp) \right] }^{\mathrm{d}_{\vyp} \Phi},\overbrace{
\sigma h_{x} }^{\mathrm{d}_{y_n}\Phi}},
\end{equation}
where $\nabla q = (q_2,\ldots,q_{n-2})^T$, and we use the convention
$q_i = q_{y_i}$.

Let \bel{PiL equal} \Pi_L(s,(\vy'',y_n),({\vx^{(1)}}',x^{(1)}_n),\sigma_1) =
\Pi_L(s,(\vy'',y_n),({\vx^{(2)}}',x^{(2)}_n),\sigma_2).\ee Using
\eqref{hs cond} and that
the $\dd_s \Phi$ and the $\dd_{\yn}\Phi$ components in \eqref{Pi_L}are equal, we see
\begin{equation}\label{hx/hs}
\frac{h_{x}}{h_s}\paren{s,x^{(1)}_n - y_n} =
\frac{h_{x}}{h_s}\paren{s,x^{(2)}_n - y_n} \implies x_n = x^{(1)}_n = x^{(2)}_n,
\end{equation}
by condition \eqref{hx/hs injective}. Further $\sigma_1 = \sigma_2$
since $h_s \neq 0$ by condition \eqref{hs cond}.

Setting the $\dd_{\vyp}$ components of $\Pi_L$ equal, we see
\bel{dy''}\left[ \nabla q, I \right]\paren{{\vx^{(1)}}'-{\vx^{(2)}}'}
= 0.\ee The rows of $A = \left[ \nabla q, I \right]$ span a plane
parallel to the plane tangent to $Q$ at $(q,\vyp))$, and \[\nul(A) =
\{ \nu (1,-\nabla q)^T : \nu\in \mathbb{R}\}.\] Thus, ${\vx^{(1)}}'$
and ${\vx^{(2)}}'$ must lie on a line parallel to $(1,-\nabla q)^T$.
As ${\vx^{(1)}}'$ and ${\vx^{(2)}}'$ are also on the same sphere
centered at $(q,\vyp)$ of radius $h(s,\xn-\yn)$, either ${\vx^{(1)}}'$
and ${\vx^{(2)}}'$ are equal or they are the reflections of one
another in the tangent plane to $Q$ at $(q,\vyp)$. Thus, $\vx_1$ and
$\vx_2$ are either equal or the reflections of one another in the
plane tangent to $S$ at $(\vyp,\yn)$. By our assumption \eqref{tangent
plane}, ${\vx^{(1)}}'={\vx^{(2)}}'$ and $\Pi_L$ is injective above
$D$.

% \tc{Note that the columns of $\Upsilon$ are orthogonal and nonzero
% (since $\Theta$ is not a pole) but not orthonormal in general because
% the tangent vectors $\partial\Theta(\dalpha)/\partial \alpha_j$
% defined by spherical coordinates are not all orthonormal.  However,
% they are orthogonal to
% each other and, of course to $\Theta(\dalpha)$. This doesn't change
% the proof since the matrix $[\Theta,\Upsilon]$ still has orthogonal
% columns and that's what is needed for getting $\Theta(\dalpha)$ to be
% normal to $(1,-\nabla q)^T$.}

%\tc{This is a super minor point, but if we choose coordinates before
%knowing $\vx^{(1)}$ and $\vx^{(2)}$, then it is possible that $\vx^{(2)}$ is a pole of $S^{n-2}$ and the coordinates don't work. We would need to note that this argument is valid even if $\Theta^{(2)}$ is a pole of $S^{n-2}$ in these coordinates. \tred{Do we care?  if so, we can say something like, ``WLOG we may assume $\vx^{(1)}$ and $\vx^{(2)}$ are not parameterized by a  pole in these coordinates."}}

For $\vx\in \Rc(s,(\vyp,\yn))$, $(x_1,\vxp)$ lies on an $n-2$
dimensional sphere of radius $\sqrt{h}(s,x_n-y_n)$. Let us parameterize this
sphere using standard spherical coordinates $\dalpha =
(\alpha_1,\ldots,\alpha_{n-2})$. Then, \bel{def:Theta} (x_1,\vxp) =
(q(\vyp),\vyp) + \sqrt{h}(s,x_n-y_n)\Theta(\dalpha) = (q,\vyp) +
\sqrt{h}\Theta,\ee where $\Theta=\Theta(\dalpha) \in S^{n-2}$ is not a
pole of the sphere. Now, equation \eqref{Pi_L} becomes
\begin{equation}
\label{Pi_L_1}
\Pi_L(s,(\vy'',y_n),\dalpha,x_n,\sigma) = \paren{ s, (\vy'',y_n), \overbrace{-\sigma
h_s}^{\mathrm{d}_s\Phi}, \overbrace{-2\sigma\sqrt{h}\left[ \nabla q, I \right]\Theta^T }^{\mathrm{d}_{\vyp}},\overbrace{ \sigma h_{x} }^{\mathrm{d}_{y_n}\Phi} },
\end{equation}
where $I = I_{(n-2) \times (n-2)}$. Since $h>0$ on $\Rc(s,(\vyp,\yn))$, these coordinates are
smooth.

 The differential of $\Pi_L$ is
\begin{equation}\label{DPi_L}
D\Pi_L = \kbordermatrix {&\mathrm{d}s,\nabla_{(\vy'',y_n)} & \mathrm{d}\sigma & \mathrm{d}x_n & \nabla_{\dalpha} \\
s, (\vy'',y_n) & I_{n\times n} & \textbf{0}_{n \times 1} & \textbf{0}_{n \times 1} & \textbf{0}_{n \times (n-2)} \\
\mathrm{d}_s\Phi & \cdot & -h_s & -\sigma h_{s x_n} & \textbf{0}_{1 \times (n-2)}\\
\mathrm{d}_{y_n}\Phi  & \cdot & h_{x} & \sigma h_{xx} & \textbf{0}_{1 \times (n-2)} \\ 
\nabla_{\vyp}\Phi & \cdot & \cdot & \cdot & -2\sigma \sqrt{h}\left[ \nabla q, I \right] \Upsilon,}
\end{equation}
where $\Upsilon =
[\Theta^T_{\alpha_1},\ldots,\Theta^T_{\alpha_{n-2}}]$, and
$[\Theta^T,\Upsilon]$ is a matrix with nonzero orthogonal columns.

We
have
\begin{equation}
\label{h_det}
\det \begin{pmatrix}  -h_s & -\sigma h_{s x_n} \\
h_{x} & \sigma h_{xx} \end{pmatrix} = -\sigma \paren{ h_sh_{xx} - h_{x}h_{s x_n} } = -\sigma h^2_s \cdot \frac{\mathrm{d}}{\mathrm{d}x_n}\paren{ \frac{h_{x}}{h_s} } \neq 0,
\end{equation}
by conditions \eqref{hs cond} and \eqref{hx/hs neq 0}. 

If $\det\paren{\left[ \nabla q, I \right] \Upsilon} = 0$, then some
linear combination of the columns of $\Upsilon$ is in $\nul(A)$, so
$\nu(1,-\nabla q)^T$ is in the span of the columns of $\Upsilon$.
Since $\Theta^T$ is normal to $\spann(\Upsilon)$,
$\Theta^T$ must be normal to $  (1,-\nabla q)^T$, which means
$\vx$ is on the plane tangent to $S$ at $(\vyp,\yn)$. This is not
possible because of \eqref{tangent plane}. Therefore, $\Pi_L$ is an
immersion.

%new:

Now, assume $R$ satisfies the Bolker condition above $D$, and $h$
satisfies \eqref{hs cond}. Then, $R$ is, by definition, a FIO. Because
$\Pi_L$ is an immersion, $\Pi_R$ is also an immersion \cite{Ho1971},
so the phase function $\Phi$ is nondegenerate.

Regarding \eqref{hx/hs injective}, if $h_{x}/h_s$ is not injective
then there are $\sigma_1, \sigma_2, x_n^1, x_n^2$ not all equal to
each other such that the third and last components in \eqref{Pi_L} are
equal. Since $h_s$ is never zero, this means that equality holds in
the left-hand side of \eqref{hx/hs}.   So, if Bolker holds then
\eqref{hx/hs injective} is true.

% Using  the argument in the paragraph with \eqref{hx/hs}, we see that
% if there are two  points
% $\vx^{(1)}$ and $\vx^{(2)}$ in $\Rc(s,(\vyp,\yn)$ such that
% $\frac{h_{x}}{h_s}\paren{s,x^{(1)}_n - y_n} =
% \frac{h_{x}}{h_s}\paren{s,x^{(2)}_n - y_n}$, then $x^{(1)}_n $ must
% equal $x^{(2)}_n$ so \eqref{hx/hs injective} holds.

Let $\vx\in \Rc(s,(\vyp,\yn))$ and let $\vx_m$ be its reflection in
the tangent plane, $T$,to $S$ at $(q,\vyp,\yn)$. If $\vx \neq \vx_m$ and
$\vx_m\in D$, then the argument around \eqref{dy''} shows that $\Pi_L$
is not injective. This proves \eqref{tangent plane} for points not on
$T$.

If Bolker holds, the matrix in\eqref{DPi_L}, $[\nabla q, I]\Upsilon$,
has maximum rank, so $\spann(\Upsilon)$ contains no nonzero vector in
$\nul(A)$. Therefore, $\Theta^T$ in \eqref{def:Theta} is not
perpendicular to $(-1,\nabla q)^T$ and $\vx$ is not in the tangent
plane to $S$ at $(s,(\vyp,\yn))$. This shows \eqref{tangent plane} for
points on $T$. For the same reason, \eqref{h_det} must be valid since
$h_s\neq 0$. This shows that \eqref{hx/hs neq 0} holds. \end{proof}

\section{An Inversion Framework}
\label{inversion}

We will present two different cases in which inversion of the surface of rotation Radon transform is possible by first taking a Fourier transform in the vertical coordinate. 

\subsection{Analytic inversion formula for the cone transform}
\label{cone_formula}

In this section, we consider the case when $h(s,x) = sx$, and $\mathcal{R}(s,\vy)$ is a cone. In this case, we have the alternate expression for $R$
\begin{equation}
\label{R_C}
Rf(s,\vy) = \sqrt{1+s^2}\int_{0}^{\infty} \int_{\Theta\in S^{n-2}}t^{n-2}f\paren{ t\Theta + \vy',y_n - st} \mathrm{d}S^{n-2}\mathrm{d}t,
\end{equation}
where $s\in\mathbb{R}$. $Rf$ defines the integral of $f$ over a cone with gradient $s$, vertex $\vy \in S$, and axis of revolution parallel to $x_n$ which is contained in $S$. 

Next, we will present the theorem which proves injectivity of the cone transform. For the proof we will require the spherical Radon transform
$$M : L^2_c(\mathbb{R}^{n-1}) \to L^2\paren{\mathbb{R}_+\times \mathbb{R}^{n-1}}$$ 
For $f \in L^2_c(\mathbb{R}^{n-1})$, $Mf(t,\vy')$ defines the
integral of $\hat{f}_{\xi} \in L^2_c(\mathbb{R}^{n-1})$ over the $(n-2)$-dimensional sphere with radius $t$ and center $\vy'$.

%The proof follows similar steps as in the proof of Theorem \ref{thm:inversion}.

\begin{theorem}
\label{thm:inversion1}
Suppose that $S = Q\times \mathbb{R}$ is a real-analytic manifold and $Q$ is 
the boundary of an open convex set $\Sigma \subset \mathbb{R}^{n-1}$. Let  $f\in L^2_c(\Sigma \times \mathbb{R})$ and $h(s,x) = sx$. Under these conditions, $Rf =
0 \implies f = 0$.
\end{theorem}

\begin{proof}
After taking the Fourier transform in $y_n$ on both sides of \eqref{R_C}, we have,
\begin{equation}
\widehat{Rf}(s,\vy',\xi)  = \sqrt{1+s^2}\int_{0}^{\infty} e^{-i st\xi} M\hat{f}_{\xi}(t,\vy')\mathrm{d}t
\end{equation}
where $\hat{f}_\xi$ is the Fourier transform in the $x_n$ component evaluated at $\xi$. Let $s' = s\xi$. Then, for $\xi \neq 0$, we have
\begin{equation}
\label{4.9}
M\hat{f}_{\xi}(t,\vy') = \frac{1}{2\pi}\int_{-\infty}^{\infty} \frac{ \widehat{Rf}\paren{\frac{s'}{\xi},\vy',\xi} }{\sqrt{1+\paren{\frac{s'}{\xi}}^2}} e^{i s' t} \mathrm{d}s'
\end{equation}
for all $t\geq 0$ and $\vy' \in Q$. We have now established a link between cone transform \eqref{R_C} and the spherical Radon transform. %In the proof of Theorem \ref{thm:inversion}, we solved a Volterra equation to recover $M\hat{f}_{\xi}$. Here we take two partial Fourier transforms (one in the $y_n$ variable, and the other in $s'$) to recover $M\hat{f}_{\xi}$. 

Let $Rf = 0$. Then, by \eqref{4.9}, $M\hat{f}_{\xi}(t,\vy') = 0$
for any $t\geq 0$, $\vy'\in Q$, and $\xi \neq 0$.  Furthermore,
$\hat{f}_\xi$ is supported in $\Sigma$ and its boundary, $Q$ is
convex and real-analytic. This, plus
results of \cite{MSaloTzouSupport} on analytic regularity under the
Bolker assumption, and the arguments in the proof of \cite[Corollary
3.2]{Q2006:supp} for the spherical transform show that, since $M\hat{f}_{\xi} = 0$ for all
centers on $Q$,  $\hat{f}_{\xi}=0$, for any $\xi \neq 0$.
As $f$ is of compact support, the $\xi$ component of $\hat{f}_{\xi}$,
is an entire analytic function, and thus $\hat{f}_{\xi}=0$ for all
$\xi\in\mathbb{R}$. Therefore, $f = 0$, and this completes the proof.
\end{proof}

%Since $f$ is supported away from $S$, $\hat{f}_{\xi}$ is equal to zero
%on all spheres of sufficiently small radius centered on $Q$.

\begin{discussion}
The function, $\widehat{Rf}(s'/\xi,\vy',\xi)/\sqrt{1+(s'/\xi)^2}$, is
the Fourier transform of $M\hat{f}_{\xi}$ in the $t$ variable at
value $s'$. If $f$ is compactly supported, then
$M\hat{f}_{\xi}(t,\vy')$ is a compactly supported function of $t$ for
any fixed $\vy'$. Thus, the Fourier transform in $t$ is an entire
analytic function and, by analytic continuation, we need only the
$s\in \Omega$, for some open $\Omega \subset \mathbb{R}$, to
determine $\widehat{Rf}(s'/\xi,\vy',\xi)/\sqrt{1+(s'/\xi)^2}$
everywhere. Therefore, we need only that ${Rf}(s,\vy)
= 0$ for $s\in \Omega$, and $\vy'\in Q$, to show $f=0$.

%\tc{Don't we need this only for $\vy'\in Q$, not the original $\vy'\in
%\rnm$? It's a cute proof.}

In cases when $M$ has an explicit left inverse, $M^{-1}$, e.g., when $Q$ is a sphere and $S$ is cylinder, and we know $Rf$ for all $s\in \mathbb{R}$, and $\vy \in S$, then, using \eqref{4.9}, we have the explicit expression for $f$ in terms of $Rf$
\begin{equation}
f = \frac{1}{2\pi}\mathcal{F}^{-1}_{\xi}M^{-1}\paren{\int_{-\infty}^{\infty} \frac{ \widehat{Rf}\paren{\frac{s'}{\xi},\vy',\xi} }{\sqrt{1+\paren{\frac{s'}{\xi}}^2}} e^{i s' t} \mathrm{d}s'},
\end{equation}
where $\mathcal{F}^{-1}_{\xi}$ is the inverse Fourier transform in the $\xi$ variable.
\end{discussion}

In this section, we presented injectivity results for $R_{\mu}$, and for $R$ in the case when $h(s,x) = sx$. The above discussion provides an explicit inverse for the cone transform in the special case when $M$ has an explicit left inverse. In the next section, we consider when $h$ is an even function with a single maximum,
as well as a generalisation of this case.
%Later, in appendix \ref{appA} , we provide further injectivity results and inversion methods for a class of Radon transforms which define the integrals over surfaces of revolution of smooth bounded curves.

\subsection{Analytic inversion in the case of symmetric curves}

In this section, we provide a general inversion framework for surface
of revolution Radon transforms where the surfaces have some symmetry.
Throughout this section, we let \[B^n_{r}(\vy) = \{\vx \in
\mathbb{R}^n : |\vx-\vy|< r\}\]denote the open unit ball in $\mathbb{R}^n$ of radius $r$. We will consider the same geometry as in section \ref{bolker} but will reformulate the operator and add some hypotheses on the function $h$. Indeed, we suppose that $C = \{s>0\}$ and for every $s$,  $\Omega_{h,s} = \{x \ : \ (x,s) \in \Omega_h\}$ is a symmetric interval around the origin, that $h(s,\cdot)$ is an even function on this interval, that $h\rightarrow 0$ at the boundaries of this interval, that $h(s,\cdot)$ takes a single maximum at $x = 0$ where the second derivative $h_{xx}$ does not vanish, that otherwise the $h_x$ does not vanish, and that $h(s,0)$ is an increasing function with $h(s,0) \rightarrow 0$ as $s\rightarrow 0$. There are several examples of practical interest where $h$ satisfies these hypotheses and we discuss them in section \ref{examples}. In this case, equation \ref{eq:Radon} is equivalent to
\begin{equation}\label{eq:Radon2}
Rf(s,
\vy) = \int_{\Omega_{h,s}}  \int_{\Theta \in S^{n-2}} h(s,x)^{\frac{n-2}{2}} f( \sqrt{h}(s,x) \Theta + \vy', y_n+x) \sqrt{1+ \frac{h_x^2}{4h}}\ \dd S^{n-2}\dd x
\end{equation}
where $\dd S^{n-2}$ is the surface measure on $S^{n-2}$. By the hypotheses on $h$, $\Omega_{h,s}$ can be split into two intervals $(-b(s),0)$ and $(0,b(s))$, and $h(s,\cdot)$ is invertible on each of these intervals. In fact, because $h(s,\cdot)$ is even there will be a function $\mu(s,\cdot): (0,h(s,0)) \rightarrow (0,b(s))$ such that
\[
h(s,\mu(s,t)) = h(s,-\mu(s,t)) = t^2.
\]
Then, changing coordinates $x = \pm\mu(s,t)$ in the inner integral of \eqref{eq:Radon2}, we have
\[
Rf(s,\vy) =\sum_{k=0}^1 \int_{0}^{\sqrt{h}(s,0)}   \int_{\Theta \in S^{n-2}} t^{n-2} f\big (t \Theta+\vy',(-1)^k\mu(s,t)+y_n\big ) \sqrt{1+ \mu^2_t} \ \dd S^{n-2}\dd t.
\]
By making a change of variables $\tilde{s} = \sqrt{h}(s,0)$, which is possible by the hypotheses, this becomes
\begin{equation}\label{eq:Radon3}
Rf(\tilde{s},\vy) =  \sum_{k=0}^1 \int_{0}^{\tilde{s}} \int_{\Theta \in S^{n-2}} t^{n-2} f\big (t \Theta+\vy',(-1)^k\mu(s,t)+y_n\big ) \sqrt{1+ \mu^2_t} \ \dd S^{n-2}\dd t.
\end{equation}
We will show how to invert a transform slightly more general than \eqref{eq:Radon3} by first taking the Fourier transform in the $y_n$ variable. Our generalisation is the following.

%Let $n\geq 3$.
Let $\mu_j = \mu_j(s,t) \in C(\mathcal{T})$, for $1\leq j \leq m$, be a set of functions which define $m$ one-parameter families of curves, where $\mathcal{T} = \{(s,t) : a\leq t \leq s \leq b\}$, for some $b>a\geq 0$. Then, we define the Radon transform
\begin{equation}
\begin{split}
\label{R_mu}
R_{\mu}f(s,\vy) = \sum_{k=0}^1\sum_{j=1}^m\int_{a}^s g_j\int_{\Theta\in
S^{n-2}}t^{n-2}f\paren{ t\Theta +
\vy',(-1)^k\mu_j(s,t)+y_n } \mathrm{d}S^{n-2}\mathrm{d}t,
\end{split}
\end{equation}
where $g_j  = g_j(s,t) = \sqrt{1+\mu_{jt}^2}$, $\mu_{jt} = \frac{\mathrm{d}}{\mathrm{d}t}\mu_j$. %, and $\mathrm{d}S^{n-2}$ is the surface measure on $S^{n-2}$. $q\in C^\infty(\Omega,\rr)$ is as defined in the proof of Theorem \ref{thm:bolker}.
%The function $\gamma : \Omega\to \mathbb{R}^{n-1}$ defines $Q$, where $\Omega \subset \mathbb{R}^{n-2}$, in the sense that $\gamma(\Omega) = Q$. For example, if $n=3$ and $\gamma(\theta) = (\cos\theta,\sin\theta)$, and $\Omega = [0,2\pi]$, then $\gamma(\Omega) = Q = S^{1}$. 
The transform, $R_{\mu}$, maps $f$ to its integrals over the surfaces of revolution of symmetric continuous curves defined by the $\mu_j$ and is equal to $R$ defined in section \ref{bolker}, as can be seen in \eqref{eq:Radon3}, for the case when $m=1$, $f=0$ on $\cup_{\vs\in S}B^n_{a}(\vs)$ (i.e., $f$ is zero up to distance $a$ from $S$), and $h$ is even in $x$.
%The centers of the surfaces of integration lie on a generalized cylinder surface, $S = Q\times \mathbb{R}$, as in section \ref{bolker}.% In certain examples of interest, $R_{\mu}f = Rf$, where $R$ is defined as in section \ref{bolker}. 
For example, when $f=0$ on $\cup_{\vs\in S}B^n_{a}(\vs)$, $m=1$ and $\mu_1(s,t) = \sqrt{s^2-t^2}$, and $h(s,x) = s^2-x^2$,
then $R_{\mu}f = Rf$ defines the integrals of $f$ over spheres, radii
$s$ with center $\vy \in S$. The reason we introduce the $\mu_j$ in
this section to define the surfaces of rotation, is for elegance and
ease of calculation in the proofs of our main injectivity theorem
(Theorem \ref{thm:inversion}), presented later in this section. We take a sum of $m$ integrals in \eqref{R_mu} to keep the discussion more general. The $m=1$ case is of the most practical interest, and later, in section \ref{examples}, we discuss further examples of $\mu_j$ and $h$, and how they apply to CST, ECST, and URT.

Taking the Fourier transform in $y_n$ on both sides of \eqref{R_mu}, yields
\begin{equation}
\label{volt_0}
\begin{split}
\widehat{R_{\mu}f}(s,\vy',\xi) &= \int_a^s\sum_{j=1}^m g_j\cos(\xi \mu_j) \left[ \int_{\Theta\in S^{n-2}}t^{n-2}\hat{f}_{\xi}\paren{ t\Theta + \vy'}\mathrm{d}S^{n-2} \right] \mathrm{d}t\\
& = \int_a^s\tilde{K}_{\xi}(s,t)M\hat{f}_{\xi}(t,\vy') \mathrm{d}t,\\
\end{split}
\end{equation}
where
\begin{equation}
\tilde{K}_{\xi}(s,t) = \sum_{j=1}^m g_j(s,t) \cos\paren{ \xi \mu_j(s,t) },
\end{equation}
$\xi$ is dual to $x_n$, $\hat{f}_{\xi}(\vx') = \hat{f}(\vx',\xi)$, and $M$ is the spherical Radon transform introduced in section \ref{cone_formula}. 
%$$M : L^2_c(\mathbb{R}^{n-1}) \to L^2\paren{\mathbb{R}_+\times \mathbb{R}^{n-1}}$$
%is a spherical Radon transform. $M\hat{f}_{\xi}(t,\vy')$ defines the
%integral of $\hat{f}_{\xi} \in L^2_c(\mathbb{R}^{n-1})$ over the $(n-2)$-dimensional sphere with radius $t$ and center $(q(\vy''),\vy'')$. 
Equation \eqref{volt_0} is a Volterra equation of the first kind,
which we aim to solve for $M\hat{f}_{\xi}$. Then we use known results
on $M$ \cite{Q2006:supp, MSaloTzouSupport} to derive injectivity conditions on $R_{\mu}$.

We now have our theorem, which provides injectivity conditions on $R_{\mu}$.

\begin{theorem}
\label{thm:inversion}  
Let $b>a\geq 0$, and let $\mathcal{T} = \{(s,t) : a\leq t \leq s \leq b\}$. Suppose that
$Q$ is the boundary of an open convex set $\Sigma \subset \mathbb{R}^{n-1}$. Let $f\in L^2_c(\Sigma \times \mathbb{R})$.

%Let $f\in L^2_c(\mathbb{R}^n)$, with $f = 0$ on $\cup_{\vy\in
%S}B^n_{a}(\vy)$, where $a>0$ and
%$B^n_{r}(\vy) = \{\vx \in \mathbb{R}^n : |\vx-\vy|\leq r\}$ is the
%$n$-dimensional ball radius $r \geq 0$, center $\vy \in
%\mathbb{R}^n$.
Consider the conditions
\begin{enumerate}
\item\label{(0)}
$f=0$ at all points of distance less than $a$ from $S$.

\item\label{(1)} $\mu_j$ is of the form $\mu_j(s,t) = \sqrt{s-t}\cdot \tau_j(s,t)$, where $\tau_j \in C^{\infty}(\mathcal{T})$, and $\tau_j(s,s) \neq 0$ for $s \in [a,b]$;
\item\label{(2)} $S$ is a real-analytic manifold. %, and all planes tangent to
%$S$ are disjoint from $\text{supp}(f)$. 
\end{enumerate}
Under \eqref{(1)}, \eqref{volt_0} is uniquely solvable. Under \eqref{(0)},
\eqref{(1)} and \eqref{(2)}, $R_{\mu}f =
0$ on $[a,b]\times S$ implies $f = 0$ 
at all points of distance less than $b$ from $S.$
\end{theorem}

\begin{remark}
Assumption \eqref{(1)} gives conditions on $\mu_j$ so that $M\hat{f}_{\xi}$
is uniquely recoverable from $R_{\mu}f$. For $R$ in
\eqref{eq:Radon3}, this first assumption is satisfied because the
second derivative $h_{xx}$ does not vanish at $x = 0$. Assumption
\eqref{(2)} provides injectivity conditions on $M$, as in
\cite[Theorem 2.4]{Q2006:supp}. \eqref{(1)} and \eqref{(2)} combined give injectivity conditions for $R_{\mu}$.
\end{remark}

\begin{remark}
It is possible to prove an extension of Theorem \ref{thm:inversion} for operators without the symmetry condition (i.e. \eqref{R_mu} without a sum in $k$) provided the $\mu_j$ satisfy certain asymptotic compatibility conditions at the boundary of their definition. However, this is not necessary for any of the practical examples given in section \ref{examples}.
\end{remark}

\begin
{proof}[Proof of Theorem \ref{thm:inversion}]
Let us assume \eqref{(1)} holds. Given the specific form of the $\mu_j$ specified in \eqref{(1)}, we have
\begin{equation}
\begin{split}
g_j(s,t) &= \sqrt{1+\mu_{jt}^2}\\
&= \frac{\sqrt{(s-t)+\paren{(s-t)\tau_{jt} - \frac{1}{2}\tau_j}^2}}{\sqrt{s-t}}\\
&= \frac{\kappa_j(s,t)}{\sqrt{s-t}}.
\end{split}
\end{equation}

%\jc{\tred{change $h$ notation here to avoid confusion? Use $\kappa$}}

We have, $\kappa_j(s,s) = \frac{1}{2}|\tau_j(s,s)|>0$ for $s\in[a,b]$, and $\kappa_j(s,t)>0$, when $s\neq t$ and $(s,t) \in \mathcal{T}$. Therefore, since $\kappa_j$ is continuous and $\mathcal{T}$ is compact, $\kappa_j$ is bounded away from zero on $\mathcal{T}$ and $\kappa_j \in C^{\infty}(\mathcal{T})$.

Now, by \eqref{volt_0}, we have
\begin{equation}
\label{volt_1}
\begin{split}
\widehat{R_{\mu}f}(s,\vy',\xi) &= \int_a^s\frac{\sum_{j=1}^m \kappa_j\cos(\xi \mu_j) }{\sqrt{s-t}}\left[ \int_{\Theta\in S^{n-2}}t^{n-2}\hat{f}_{\xi}\paren{ t\Theta + \vy'}\mathrm{d}S^{n-2} \right] \mathrm{d}t\\
& = \int_a^s\frac{K_{\xi}(s,t)}{\sqrt{s-t}}M\hat{f}_{\xi}(t,\vy') \mathrm{d}t,\\
\end{split}
\end{equation}
where
\begin{equation}
K_{\xi}(s,t) = \sum_{j=1}^m \kappa_j(s,t) \cos\paren{ \xi \mu_j(s,t) }.
\end{equation}
We have
\begin{equation}
K_{\xi}(s,s) = \frac{1}{2}\sum_{j=1}^m |\tau_j(s,s)| > 0
\end{equation}
for $s \in [a,b]$, by \eqref{(1)}, and
\begin{equation}
\begin{split}
\frac{\mathrm{d}}{\mathrm{d}s} \cos(\xi \mu_j) &= -\xi\cdot \frac{(s-t)\tau_{js} - \tau_j}{\sqrt{s-t}}\sin(\xi \mu_j)\\
&= \phi_j \sinc(\xi \mu_j) \in C(\mathcal{T}),
\end{split}
\end{equation}
where $\phi_j = \xi^2\tau_j\paren{ (s-t)\tau_{js}-\tau_j } \in C^{\infty}(\mathcal{T})$. Thus, $\frac{\mathrm{d}}{\mathrm{d}s} K_{\xi} \in C(\mathcal{T})$, given also the smoothness of the $\kappa_j$. Further, since $M\hat{f}_{\xi} \in L^2\paren{\mathbb{R}_+\times \mathbb{R}^{n-1}}$, $\widehat{Rf}(s,\vy',\xi)$ is an absolutely continuous function of $s$, by \cite[Lemma 3.3]{Q1983-rotation}. Therefore, by \cite[Theorems A and B]{Q1983-rotation}, we can now solve the Volterra equation of the first kind \eqref{volt_1} uniquely for $M\hat{f}_{\xi}(t,\vy')$, for all $t\in[a,b]$, $\vy' \in \mathbb{R}^{n-1}$, and $\xi\in \mathbb{R}$. 

 To finish the proof, let us assume $Rf = 0$, and that \eqref{(1)} and
\eqref{(2)} hold. Using the above calculations, $Rf = 0$ implies
$M\hat{f}_{\xi}(t,\vy') = 0$ for any $t\in[a,b]$, $\vy'\in Q$, and
$\xi \in \mathbb{R}$. As $\supp(f)$ is at least distance $a$
from $S$, $M\hat{f}_{\xi}(t,\vy)=0$ for all $(t,\vy)\in [0,b)\times
Q$ and  $\hat{f}_\xi$ is supported in $\Sigma$.  Since, in addition, $Q$ is a convex
real-analytic manifold by assumption \eqref{(2)}, we can use the
analytic regularity theorem in \cite{MSaloTzouSupport} under the
Bolker assumption and the arguments in the proof of \cite[Corollary
3.2]{Q2006:supp} to show that $\hat{f}_{\xi}=0$ on $\cup_{\vy'\in
Q}B^{n-1}_{b}(\vy')$, for any $\xi \in \mathbb{R}$. Therefore, $f = 0$
on $\cup_{\vs\in S}B^n_{b}(\vs)$. This completes the proof.
\end{proof}

%, since $M\hat{f}_{\xi}(t,\vy)= 0$ for $(t,\vy)\in [0,b)\times Q$, 

%In special cases, e.g., when $S$ is a cylinder and $f$ is supported within the cylinder, $M$ has an analytic inverse and one could apply analytic inversion formulae from the literature to recover each slice $\hat{f}_{\xi}$.

\begin{remark}
Theorem \ref{thm:inversion} establishes a key link between $R_{\mu}$
and the spherical Radon transform, on which there exists a wealth of
literature on the inversion properties. After taking the Fourier
transform in $y_n$ and inverting a 1-D Volterra operator in the first
stages of the proof of Theorem \ref{thm:inversion}, $R_{\mu}$ reduces
to a spherical Radon transform, $M$, which is applied to slices in Fourier space, $\hat{f}_{\xi}$. In special cases, this connection to the spherical transform can be used to derive methods for inverting $R_{\mu}$. For example, if $Q = S^{n-2}$ is a sphere and $f$ is compactly supported on the interior of $S = Q\times \mathbb{R}$, then $\hat{f}_{\xi}$ is compactly supported on the interior of $Q$, for any $\xi\in\mathbb{R}$, and $\hat{f}_{\xi}$ can be recovered explicitly from $M\hat{f}_{\xi}$ using the formulae of \cite{kunyansky2007explicit}.
\end{remark}

\section{Example applications}
\label{examples}
In this section, we give some example $h$ and $\mu_j$ which have applications in CST, URT, and ECST. We only consider here the conditions on $h$ and $\mu$ given in Theorem \ref{thm:bolker}, and Theorem \ref{thm:inversion}, respectively, which are needed in order for the Bolker condition and injectivity to hold. For the Bolker condition to hold in any of the examples given below, $\text{supp}(f)$ must  be bounded away from $S$, and $S$ must also satisfy condition \ref{tangent plane} of Theorem \ref{thm:bolker}.  For injectivity to hold, we need $\text{supp}(f)$ to be bounded away from $S$, $S$ to  be real-analytic, and $Q$ to be the boundary of a convex set, as specified in Theorem \ref{thm:inversion}.

\subsection{Elliptic arcs with fixed linear eccentricity}
\label{example:elliptic} In this example, we consider the case when
$h$ defines an elliptic arc, and $\mathcal{R}(s,\vy)$ is an ellipsoid
of revolution (or spheroid). Spheroid integral surfaces have
applications in URT \cite{WHQ}, and seismic imaging
\cite{grathwohl2020imaging}. For example, in URT, the foci of the
spheroids represent sound wave emitters and receivers, and the sound
wave travel time determines the spheroid radii. We consider the
special case when the foci of the spheroid are constrained to lie on
$S$, and the linear eccentricity of the spheroid, $c$, is fixed. This
case is of interest in the URT literature \cite{ABKQ2013} as well.

Let
\begin{equation}
\label{h_ellipse}
h(s,x) = \frac{s^2}{s^2+c^2}(s^2+c^2 - x^2) \in C^{\infty}\paren{ \Omega_h },
\end{equation}
where $s$ is the minor ellipse radius, $c$ is the fixed linear eccentricity, and $\Omega_h = \{(s,x) : s>0, -\sqrt{s^2+c^2}<x<\sqrt{s^2+c^2}\}$. Then, for fixed $s>0$, $x\to \sqrt{h}(s,x)$ defines an elliptic arc. See figure \ref{ellipses}, where we plot some example elliptic arc curves when $c=2$. 

We now aim to show that $h$ satisfies the conditions of Theorem \ref{thm:bolker}. First, it is clear that $h\to 0$ at the boundary of $\Omega_h$. Also,
\begin{equation}
h_s(s,x) = \frac{ s \paren{ c^4 - c^2(x^2-2s^2)+s^4 } }{ (s^2+c^2)^2 },
\end{equation}
which is non-zero on $\Omega_h$. Now, we have
\begin{equation}
\frac{h_{x}}{h_s} = -\frac{s x (s^2+c^2)}{c^4 - c^2(x^2-2s^2)+s^4},
\end{equation}
and it follows that
\begin{equation}
\frac{\mathrm{d}}{\mathrm{d}x}\paren{\frac{h_{x}}{h_s}} = -s(s^2+c^2)\frac{c^4 + c^2(x^2+2s^2)+s^4}{\paren{c^4 - c^2(x^2-2s^2)+s^4}^2 },
\end{equation}
which is non-zero on $\Omega_h$. Thus, the conditions of Theorem \ref{thm:bolker} are satisfied.

\begin{figure}
\centering
\begin{subfigure}{0.3\textwidth}
\includegraphics[width=0.9\linewidth, height=3.2cm, keepaspectratio]{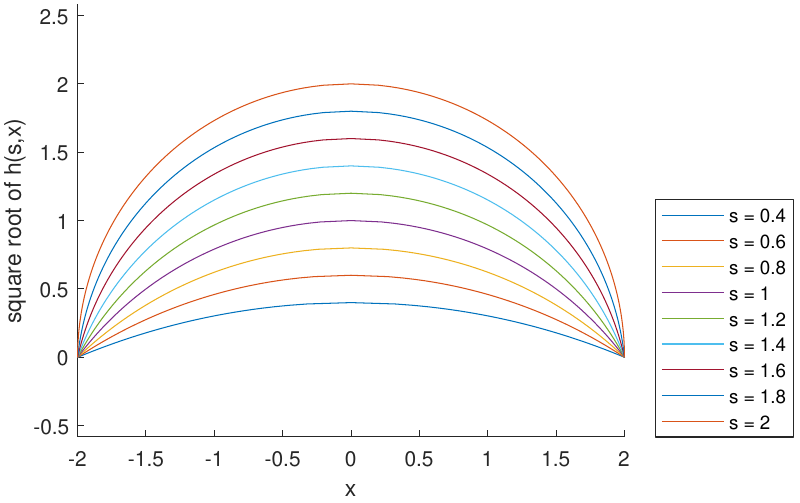}
\subcaption{circular arcs ($\alpha = 2$)} \label{arcs}
\end{subfigure}
\begin{subfigure}{0.3\textwidth}
\includegraphics[width=0.9\linewidth, height=3.2cm, keepaspectratio]{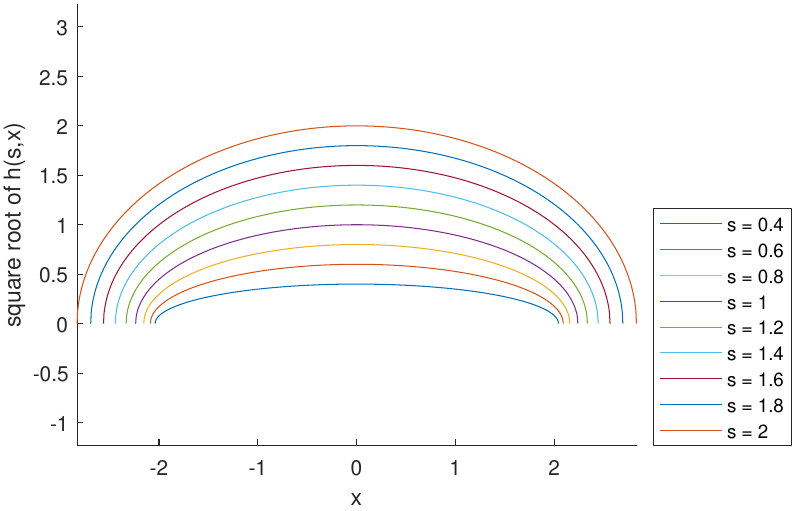}
\subcaption{elliptic arcs ($c = 2$)} \label{ellipses}
\end{subfigure}
\begin{subfigure}{0.3\textwidth}
\includegraphics[width=0.9\linewidth, height=3.2cm, keepaspectratio]{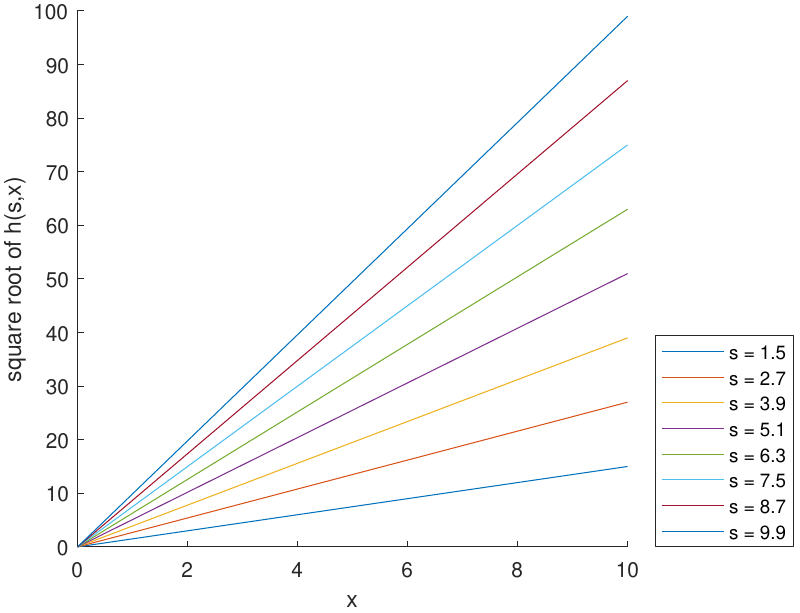}
\subcaption{lines} \label{lines}
\end{subfigure}
\caption{Example $\sqrt{h}$ mappings which define circular and elliptic arcs, and straight line curves.}
\label{F_curve}
\end{figure}

Let
$$\mu(s,t) = \sqrt{1+\paren{\frac{c}{s}}^2} \sqrt{s^2-t^2}.$$
Then, if $h$ is defined as in \eqref{h_ellipse}, $f = 0$ on $\cup_{\vs\in S}B^n_{a}(\vs)$, and $m=1$, with $\mu_1=\mu$, $Rf = R_{\mu}f$. We now aim to show the conditions on $\mu$ specified in Theorem \ref{thm:inversion} hold. We have $\mu(s,t) = \sqrt{s-t}\cdot \tau(s,t)$, where
$$\tau(s,t) = \sqrt{1+\paren{\frac{c}{s}}^2} \sqrt{s+t}.$$
Let $b>a>0$. Then, $\tau$ is smooth on $\mathcal{T} = \{(s,t) : a\leq s \leq b, a\leq t \leq s\}$. Further,
$$\tau(s,s) = \sqrt{2}\sqrt{s}\sqrt{1+\paren{\frac{c}{s}}^2} \neq 0$$
for $s\in [a,b]$. Thus, the conditions of Theorem \ref{thm:inversion} are satisfied when $m=1$, and $\mu_1 = \mu$.

\subsection{Circular arcs}\label{sect:circular}
In this example, we consider the case when $h$ defines a circular arc curve, and $\mathcal{R}(s,\vy)$ is the surface of revolution of a circular arc, which we call a lemon. A lemon is also the interior part of a spindle torus (see \cite[figure 5]{rigaud20183d}), which is a special kind of torus that self-intersects. Lemon integral surfaces have applications in CST \cite{rigaud20183d}. A lemon self-intersects at two points, which we will call the ``tips" of the lemon. In CST, the tips of the lemon represent source and detector positions, and the scattered photon energy determines the radius of the lemon. We consider the special case here where the distance between the tips of the lemon, $\alpha$, is fixed.

Let
\begin{equation}
\label{h_arc}
h(p,x) = \paren{\sqrt{\alpha^2+p^2-x^2} - p}^2 \in C^{\infty}\paren{ \Omega_h },
\end{equation}
where $\Omega_h = [0,\infty) \times (-\alpha,\alpha)$. Then, $x\to \sqrt{h}(p,x)$, for fixed $p\in [0,\infty)$, defines a circular arc. See figure \ref{arcs} for an illustration of circular arc curves when $\alpha = 2$. In figure \ref{arcs}, $s$ relates to $p$ via $s = \sqrt{\alpha^2+p^2}-p$.

We now aim to show the conditions of Theorem \ref{thm:bolker} are satisfied. 
%Note, we define $h$ in terms of $p$, as opposed to $s$,  in \eqref{h_arc}, as it makes the following calculations involving the derivatives of $h$ easier and more elegant. 
By \eqref{h_arc}, it is clear $h\to 0$ on the boundary of $\Omega_{h,p} = \{(p',x)\in \mathbb{R}^2 : p = p'\}\cap \Omega_h$, for any $p\in [0,\infty)$. We have,
\begin{equation}
h_p(p,x) = 2\paren{\frac{p}{\sqrt{\alpha^2+p^2-x^2}}-1}\paren{\sqrt{\alpha^2+p^2-x^2} - p},
\end{equation}
which is strictly less than zero on $\Omega_h$. Now,
\begin{equation}
\frac{h_{x}}{h_p} = \frac{x}{\sqrt{\alpha^2+p^2-x^2}-p},
\end{equation}
and
\begin{equation}
\frac{\mathrm{d}}{\mathrm{d}x}\paren{\frac{h_{x}}{h_p}} = \frac{(p^2+\alpha^2) - p\sqrt{p^2+\alpha^2-x^2}}{\sqrt{p^2+\alpha^2-x^2}\paren{ \sqrt{p^2+\alpha^2-x^2} - p }^2 },
\end{equation}
which is nonzero on $\Omega_h$. To see this, note
\begin{equation}
\begin{split}
(p^2+\alpha^2) - p\sqrt{p^2+\alpha^2-x^2} & \geq (p^2+\alpha^2) - p\sqrt{p^2+\alpha^2} \\
&= \sqrt{p^2+\alpha^2}\paren{ \sqrt{p^2+\alpha^2} - p } > 0.
\end{split}
\end{equation}
Thus, the conditions of Theorem \ref{thm:bolker} are satisfied.

Let
$$\mu(s,t) = \sqrt{s-t} \sqrt{\frac{st+\alpha^2}{s}}.$$
Then, if $h$ is defined as in \eqref{h_arc} and $m=1$, with $\mu_1=\mu$, $Rf(p,\vy) = R_{\mu}f\paren{\sqrt{\alpha^2+p^2}-p,\vy}$, when $p\geq 0$. We will now show that $\mu$ satisfies the conditions of Theorem \ref{thm:inversion}. We have,  $\mu(s,t) = \sqrt{s-t}\cdot \tau(s,t)$, where
$$\tau(s,t) = \sqrt{\frac{st+\alpha^2}{s}}.$$
Let $b>a>0$. Then, $\tau$ is smooth on $\mathcal{T} = \{(s,t) : a\leq s \leq b, a\leq t \leq s\}$. Further,
$$\tau(s,s) = s+\frac{\alpha^2}{s} > 0,$$
for $s\in[a,b]$. Thus, the conditions of Theorem \ref{thm:inversion} are satisfied when $m=1$, and $\mu_1 = \mu$. 
%Therefore, we can apply the theory of subsection \ref{case 1} with $m=1$ and $\mu_1 = \mu$. When $s\leq \alpha$, $Rf$ defines integrals of $f$ over lemons. When $s>\alpha$, $Rf$ defines integrals of $f$ over apples. The results of subsection \ref{case 1} hold in both cases. That is, we can use the lemon integrals to determine $M\hat{f}_{\xi}(t,\vy')$ for $t\in[0,\alpha]$. The apple integrals determine $M\hat{f}_{\xi}(t,\vy')$ for $t\in(\alpha,\infty)$. When $s>\alpha$, $\mu(s,\cdot)$ is not invertible on $[0,s]$, and we cannot define the apple curve using a function $h$. When $s\leq \alpha$, $\mu(s,\cdot)$ is  invertible on $[0,s]$ and $h(\cdot,p) = \mu^{-1}(s,\cdot)$ on $[0,\alpha]$, and $h(\cdot,p) = (-\mu)^{-1}(s,\cdot)$ on $[-\alpha,0)$, where the inverse of $\mu$ is taken in the $t$ variable. Thus, our microlocal theory does not apply to the apple case, and only applies to the lemon case in its current form (\tred{Can this be expanded easily? I feel not as the apple curve ``comes back on itself" and I've seen artifacts occur with the apple surfaces before due to this property.}).

\subsection{Straight lines}
In this example, we consider the case when $h(s,x)=sx$ defines a straight line, gradient $s$, and $\mathcal{R}(s,\vy)$ is a cone. See figure \ref{lines} for some example straight line curves. Cone integral surfaces have applications in ECST and Compton camera imaging \cite{Terzioglu-KK-cone}. In CST, the vertex of the cone corresponds to a scattering location, and the scattered photon energy determines the gradient of the cone, $s$. In this case, the injectivity of $R$ is covered in section \ref{cone_formula}. We aim to prove here that $h$ satisfies the conditions of Theorem \ref{thm:bolker}.

Let $h(s,x) = sx \in C^{\infty}(\Omega_h)$, where $\Omega_h =
\mathbb{R} \times (0,\infty)$.Then, the set $\Omega_{h,s}$ in
Theorem \ref{thm:bolker} assumption \eqref{boundary limit} is $\Omega_{h,s}=(0,\infty)$ and $h(s,\cdot)\to 0$ on the boundary
of $\Omega_{h,s}$, which in this case is $\{0\}$. We have, $h_s(s,x) = x > 0$ on $\Omega_h$. Further, ${h_{x}}/{h_s} = {s}/{x}$, and $\frac{\mathrm{d}}{\mathrm{d}x}\paren{{h_{x}}/{h_s}} = -s/x^2\neq 0$ on $\Omega_h$. Thus, the conditions of Theorem \ref{thm:bolker} are satisfied.

\section{Simulated image reconstructions}
\label{images}
In this section, we present simulated three-dimensional image reconstructions from integrals over spheres, spheroids, and lemons, which were shown to satisfy the conditions of Theorem \ref{thm:bolker} and Theorem \ref{thm:inversion} in section \ref{examples}. The surface of centers, $S$, we consider here is a cylinder. Throughout this section, $n=3$, $Q = S^1$, and $S = Q\times \mathbb{R} = \{\vx \in \mathbb{R}^3 : \sqrt{x_1^2+x_2^2} = 1\}$. The target functions, $f$, we consider in this section have compact support which is contained within the interior of $S$, and which is bounded away from $S$, which is needed for the microlocal theory of Theorem \ref{thm:bolker}, and the injectivity results of Theorem \ref{thm:inversion} to hold. $S$ is also convex and real-analytic, which is in line with the conditions of Theorem \ref{thm:bolker} and Theorem \ref{thm:inversion}. In all simulations conducted in the section, the conditions of Theorem \ref{thm:bolker} and Theorem \ref{thm:inversion} are satisfied.

\subsection{Data simulation}
\label{sim:sect}
The data is simulated using the exact model \eqref{R_mu}, with $m=1$. The definition of $\mu_1$ changes based on the integral surface, e.g., for spherical integrals we set $\mu_1(s,t) = \sqrt{s^2-t^2}$. Let $V_{\xi}$ be the Volterra operator of \eqref{volt_0}, and let $\mathcal{F}_3$ denote the partial Fourier transform in the $x_3$ variable. Then we simulate data as
\begin{equation}
\label{sim}
R_{\mu}f = \mathcal{F}^{-1}_3V_{\xi}M\mathcal{F}_3 f,
\end{equation}
where $M$ is applied to each Fourier slice, $\hat{f}_{\xi}$, of $f$, as in \eqref{volt_0}. We simulate $R_{\mu}f(s,\vy)$ for $s\in [0.2,2.2]$, and $\vy = (\cos\theta,\sin\theta,y_3) \in S$, where $\theta \in[0,2\pi]$, and $y_3 \in[-5,5]$.  For spheres, $s$ is the sphere radius. For spheroids, $s$ is the minor radius, and we set the linear eccentricity, $c=2$, as in figure \ref{ellipses}. For lemons, $s$ is the height of the lemon, and we set the distance between the lemon tips, $\alpha =2$, as in figure \ref{arcs}. After the data is generated as in \eqref{sim}, we add Gaussian noise to $R_{\mu}f$ to simulate noise.

\subsection{Inversion methods}
\label{inv:methods}
We recover $f$ from $R_{\mu}f$ by inverting the sequence of operators on the right-hand side of \eqref{sim}. Each operator is discretized and $f$ is recovered on an $N\times N\times N$ pixel grid. Throughout this section, we set $N =101$.

To apply and invert $\mathcal{F}_3$, we use the Fast Fourier Transform (FFT). To invert $V_{\xi}$, we use Tikhonov regularization, and the ``backslash" function in Matlab. This is possible as each Volterra operator, $V_{\xi}$, is one-dimensional, and hence the  discretized form is a small (in this case $N\times N$) matrix, which is stored in Matlab. We use the Landweber method, and Total Variation (TV) regularization methods to invert $M$. Specifically, to implement TV, we use the Conjugate Gradient Least Squares (CGLS) method in combination with TV denoising. The Landweber method enforces relatively weak regularization, and is included mainly to highlight some of the artifacts we would expect to see as predicted by Theorem \ref{thm:bolker}. The CGLS-TV method is included to show the effects of a more powerful regularizer.

\subsection{Delta function reconstructions}
In this section, we present reconstructions of a delta function, $\delta$, which is located on the interior of $S$. By theorem \ref{thm:bolker}, when we recover $\delta$ from its integrals over spheres, lemons, or spheroids, with centers on $S$, we would expect to see artifacts which are the reflections of $\delta$ in planes tangent to $S$. Given the convexity of $S$ in this case, and as $\delta$ is supported on the interior of $S$, the artifacts due to Bolker should also be constrained to the exterior of $S$. In figure \ref{F1}, we show reconstructions of an example $\delta$ when using the Landweber method to invert $M$, as described in section \ref{inv:methods}. On the left-hand of figure \ref{F1}, we show the artifacts due to Bolker as predicted by Theorem \ref{thm:bolker}. 
\begin{figure}
\centering
\begin{subfigure}{0.24\textwidth}
\includegraphics[width=0.9\linewidth, height=3.2cm, keepaspectratio]{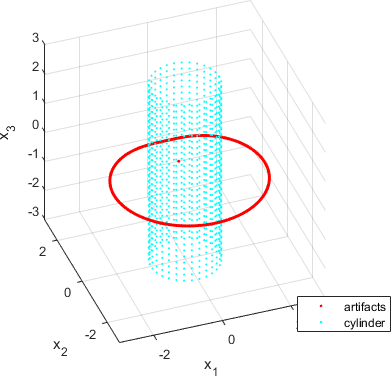}
\end{subfigure}
\begin{subfigure}{0.24\textwidth}
\includegraphics[width=0.9\linewidth, height=3.2cm, keepaspectratio]{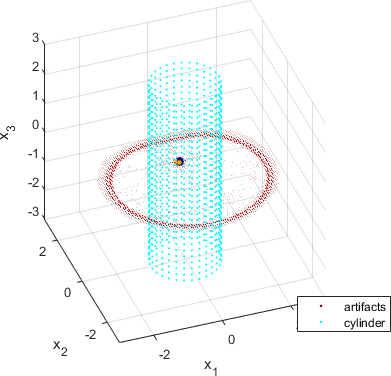}
\end{subfigure}
\begin{subfigure}{0.24\textwidth}
\includegraphics[width=0.9\linewidth, height=3.2cm, keepaspectratio]{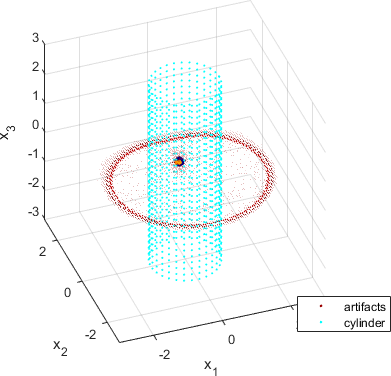}
\end{subfigure}
\begin{subfigure}{0.24\textwidth}
\includegraphics[width=0.9\linewidth, height=3.2cm, keepaspectratio]{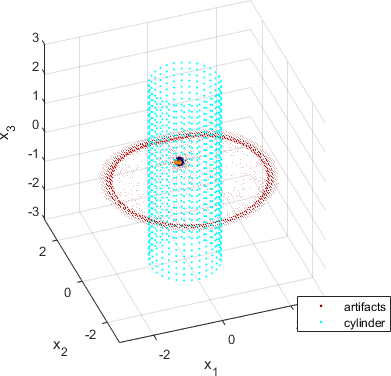}
\end{subfigure}
\begin{subfigure}{0.24\textwidth}
\includegraphics[width=0.9\linewidth, height=3.2cm, keepaspectratio]{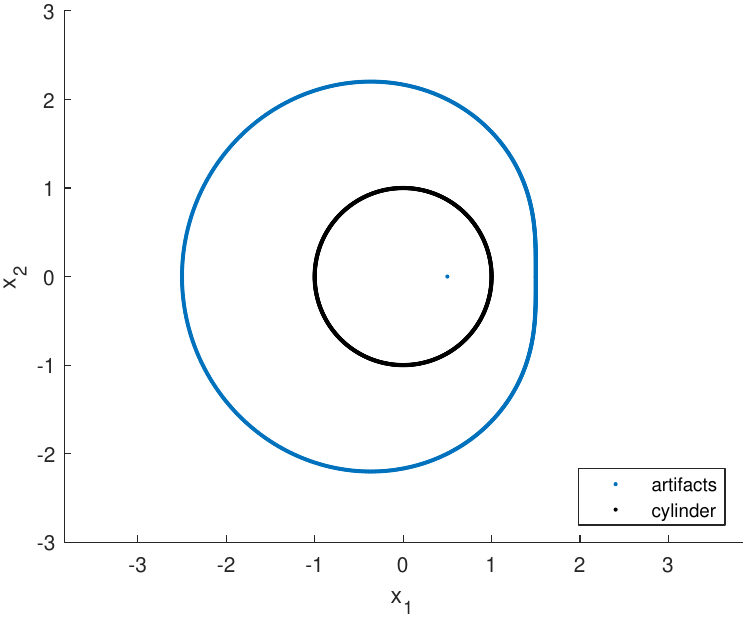}
\subcaption*{predicted}
\end{subfigure}
\begin{subfigure}{0.24\textwidth}
\includegraphics[width=0.9\linewidth, height=3.2cm, keepaspectratio]{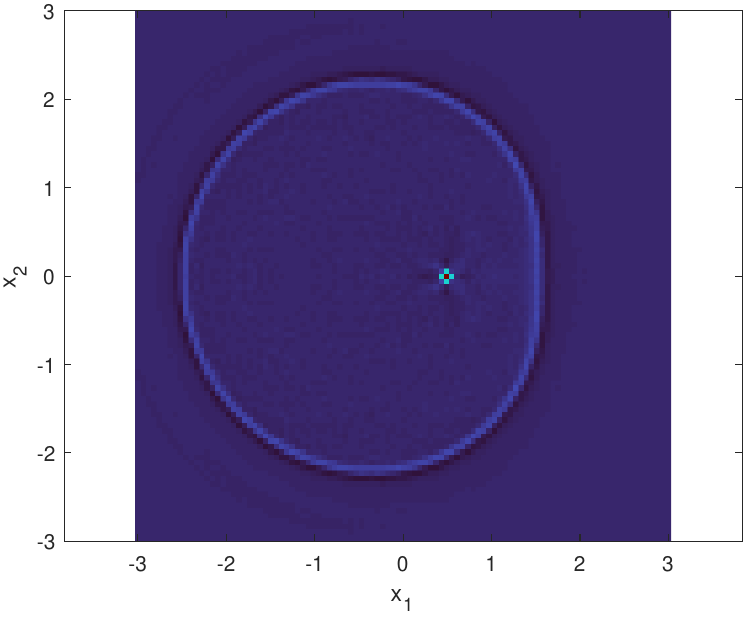}
\subcaption*{sphere}
\end{subfigure}
\begin{subfigure}{0.24\textwidth}
\includegraphics[width=0.9\linewidth, height=3.2cm, keepaspectratio]{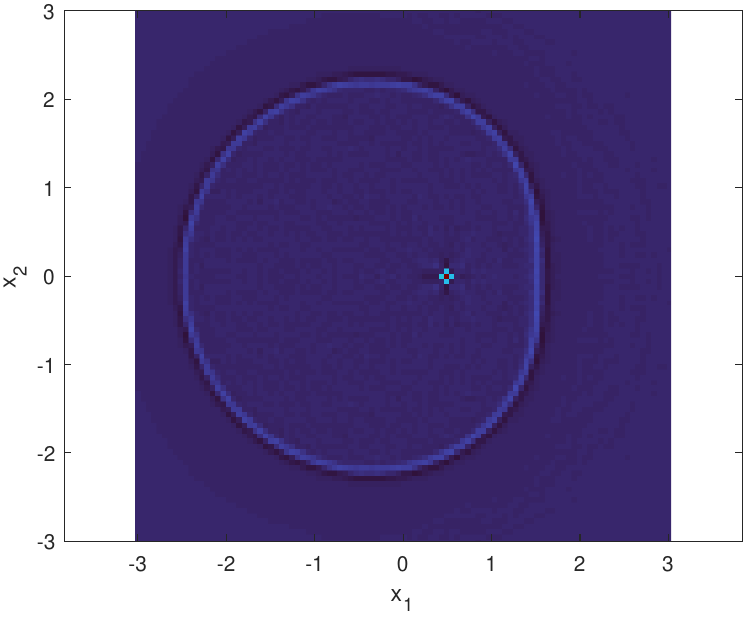}
\subcaption*{spheroid}
\end{subfigure}
\begin{subfigure}{0.24\textwidth}
\includegraphics[width=0.9\linewidth, height=3.2cm, keepaspectratio]{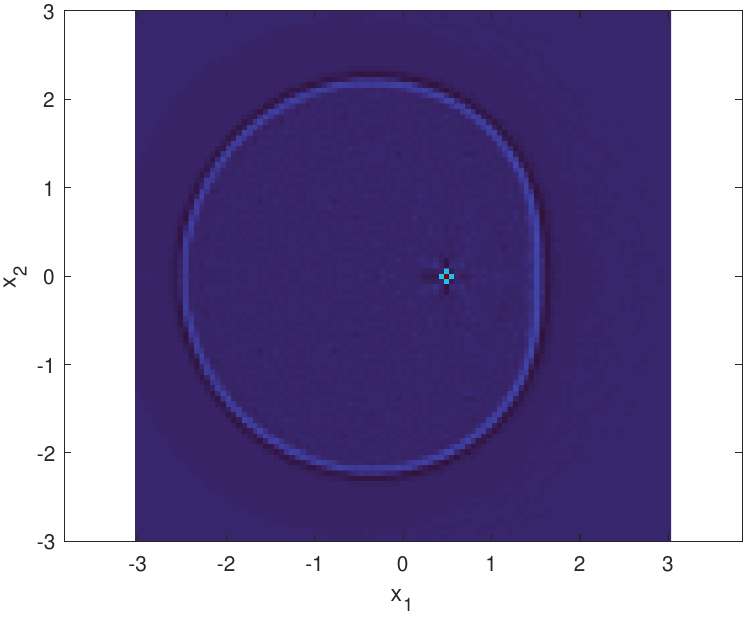}
\subcaption*{lemon}
\end{subfigure}
\caption{Predicted and observed artifacts due to Bolker for a variety of integral surfaces of revolution with centers on a cylinder. Top row - 3-D view. Bottom row - $(x_1,x_2)$ plane cross-sections. $1\%$ Gaussian noise was added in this example.}
\label{F1}
\end{figure}
In the reconstructions, the delta function is reflected in every plane tangent to $S$, which forms a cardioid type curve which is embedded in the $(x_1,x_2)$ plane. This is as predicted by our theory, and the predicted and observed artifact curves match exactly.

\subsection{Phantom reconstructions}
In this sub-section, we present simulated reconstructions of an image phantom. The phantom we consider is a hollow cuboid as pictured in the left-hand column of figure \ref{F2}. In figure \ref{geometry}, we show the size of the hollow cube phantom, and how it fits inside of $S$. We also show example curves which define the surfaces of revolution, in the case of spheres, spheroids, and lemons.
\begin{figure}[!h]
\centering
\begin{subfigure}{0.32\textwidth}
\includegraphics[width=0.9\linewidth, height=4cm, keepaspectratio]{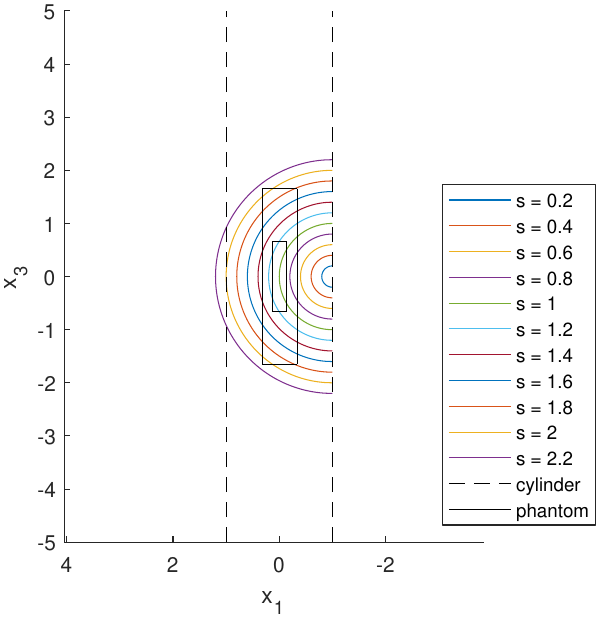}
\subcaption{spheres} \label{F1a}
\end{subfigure}
\begin{subfigure}{0.32\textwidth}
\includegraphics[width=0.9\linewidth, height=4cm, keepaspectratio]{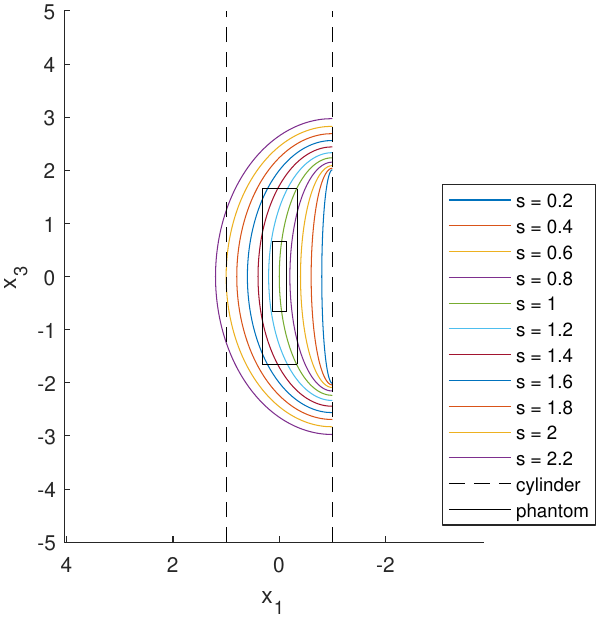} 
\subcaption{spheroids} \label{F1b}
\end{subfigure}
\begin{subfigure}{0.32\textwidth}
\includegraphics[width=0.9\linewidth, height=4cm, keepaspectratio]{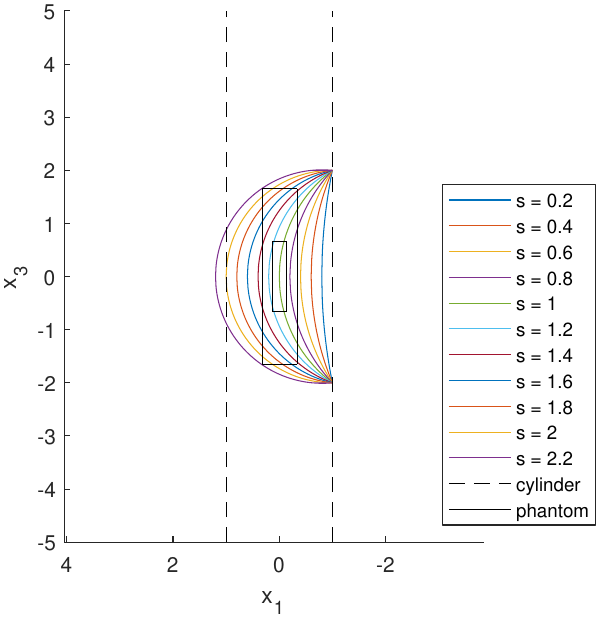}
\subcaption{lemons} \label{F1c}
\end{subfigure}
\caption{Hollow cube phantom and example curves which form the surfaces of revolution in the case of spheres, spheroids, and lemons. The $\mu_j$ and $h$ functions which define the curves shown are provided in section \ref{examples}. The sphere case in (A) is a special case of a spheroid, when $c=0$, where $c$ is the linear eccentricity, as defined in sub-section \ref{example:elliptic}.}
\label{geometry}
\end{figure}

In figure \ref{F_sino}, we show example $R_{\mu}f$ sinograms for sphere, spheroid, and lemon integral surfaces, where $R_{\mu}f$ is generated as specified in sub-section \ref{sim:sect}.
\begin{figure}
\centering
\begin{subfigure}{0.24\textwidth}
\includegraphics[width=0.9\linewidth, height=3.2cm, keepaspectratio]{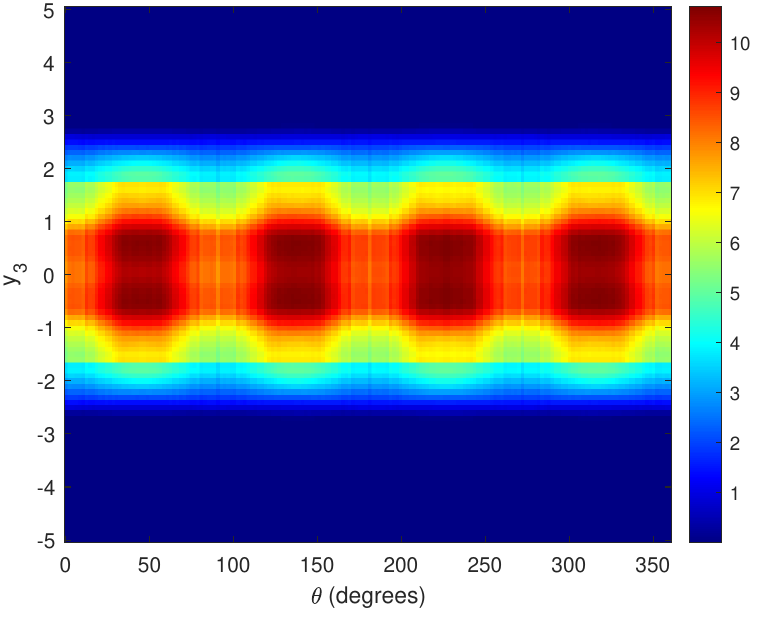}
\end{subfigure}
\begin{subfigure}{0.24\textwidth}
\includegraphics[width=0.9\linewidth, height=3.2cm, keepaspectratio]{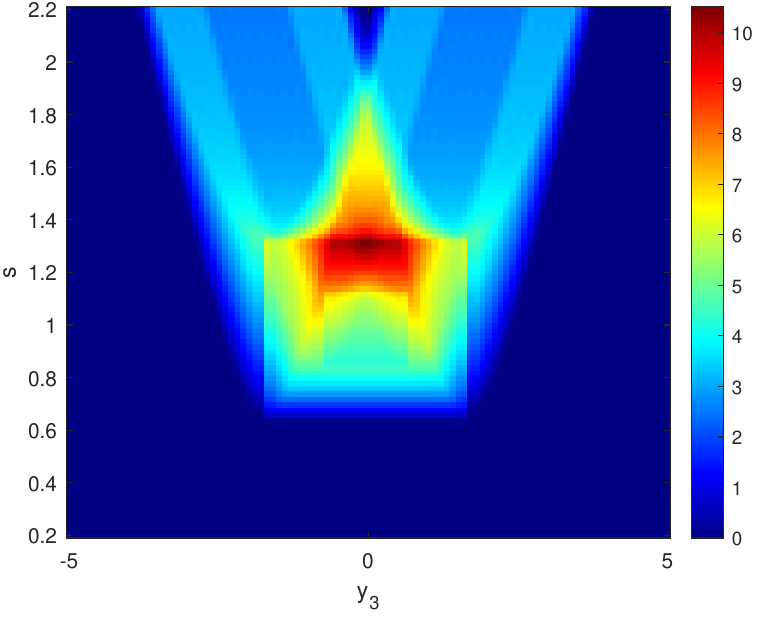}
\end{subfigure}
\begin{subfigure}{0.24\textwidth}
\includegraphics[width=0.9\linewidth, height=3.2cm, keepaspectratio]{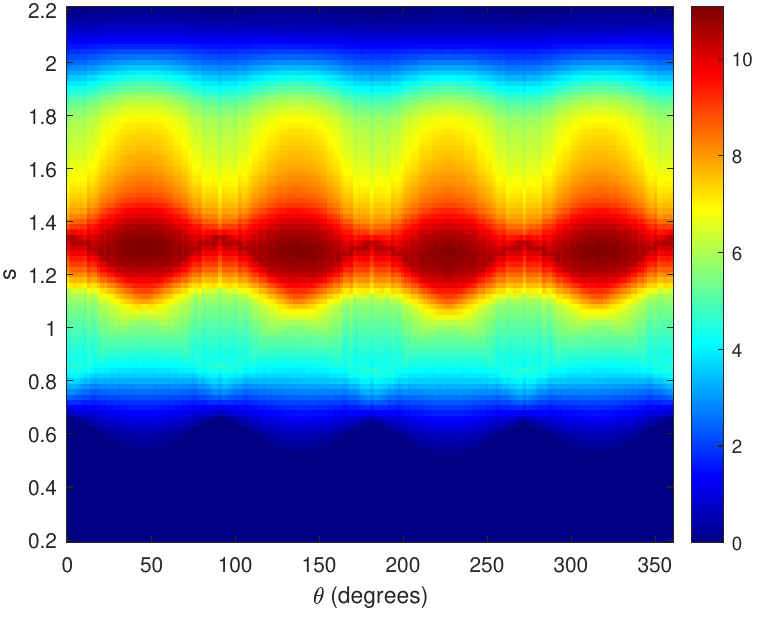}
\end{subfigure}
\\
\begin{subfigure}{0.24\textwidth}
\includegraphics[width=0.9\linewidth, height=3.2cm, keepaspectratio]{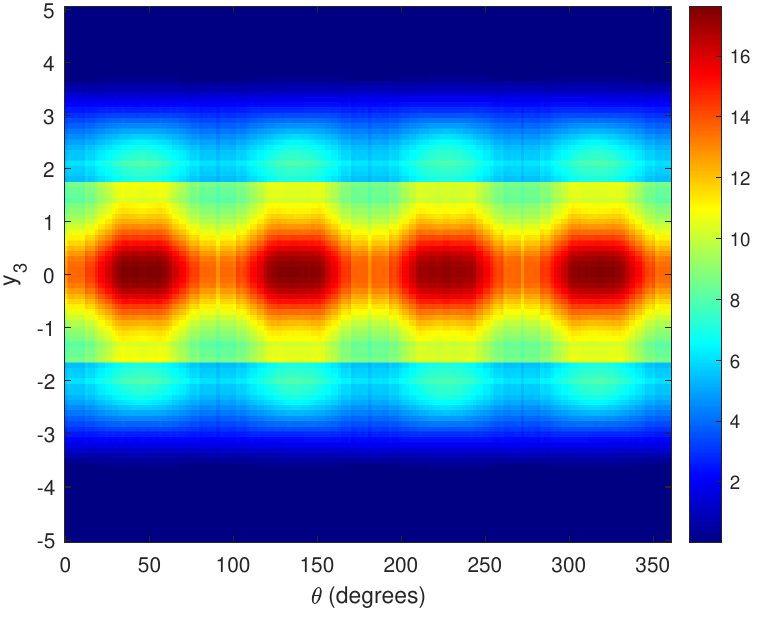}
%\subcaption{sphere}
\end{subfigure}
\begin{subfigure}{0.24\textwidth}
\includegraphics[width=0.9\linewidth, height=3.2cm, keepaspectratio]{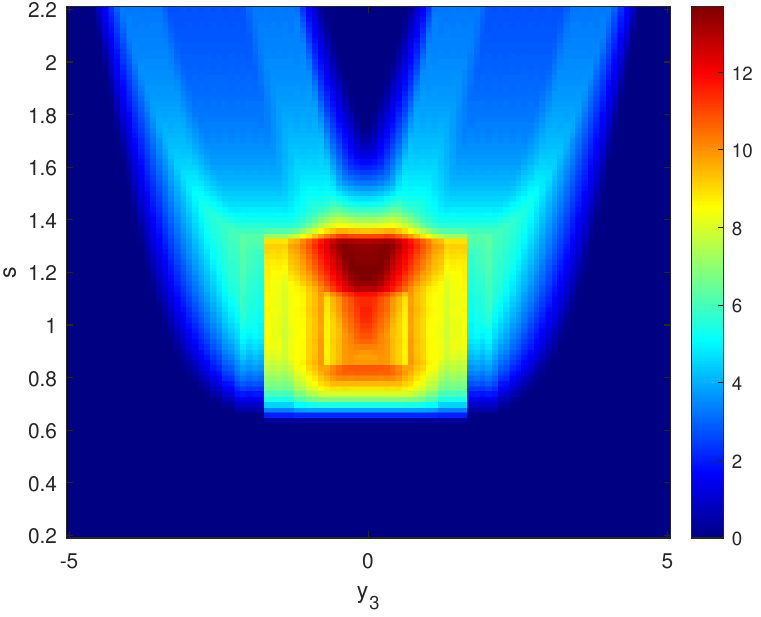}
%\subcaption{spheroid}
\end{subfigure}
\begin{subfigure}{0.24\textwidth}
\includegraphics[width=0.9\linewidth, height=3.2cm, keepaspectratio]{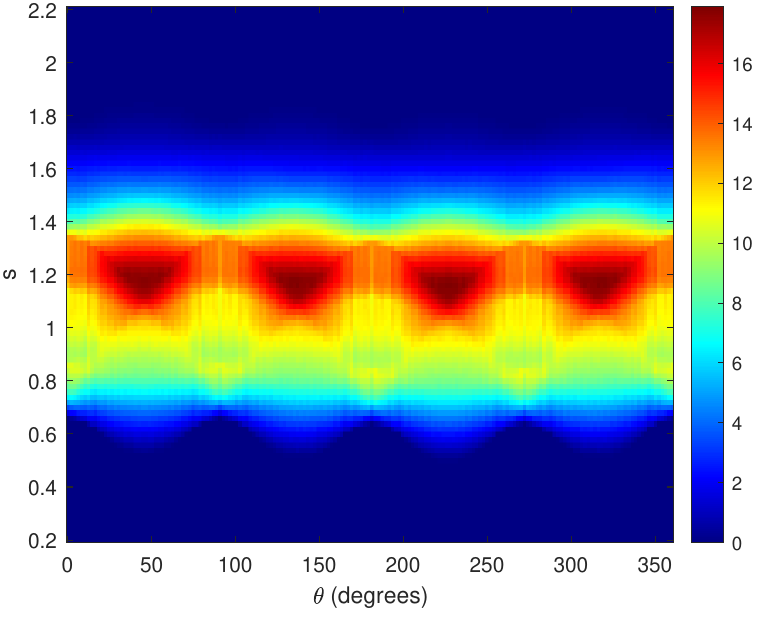}
%\subcaption{lemon}
\end{subfigure}
\\
\begin{subfigure}{0.24\textwidth}
\includegraphics[width=0.9\linewidth, height=3.2cm, keepaspectratio]{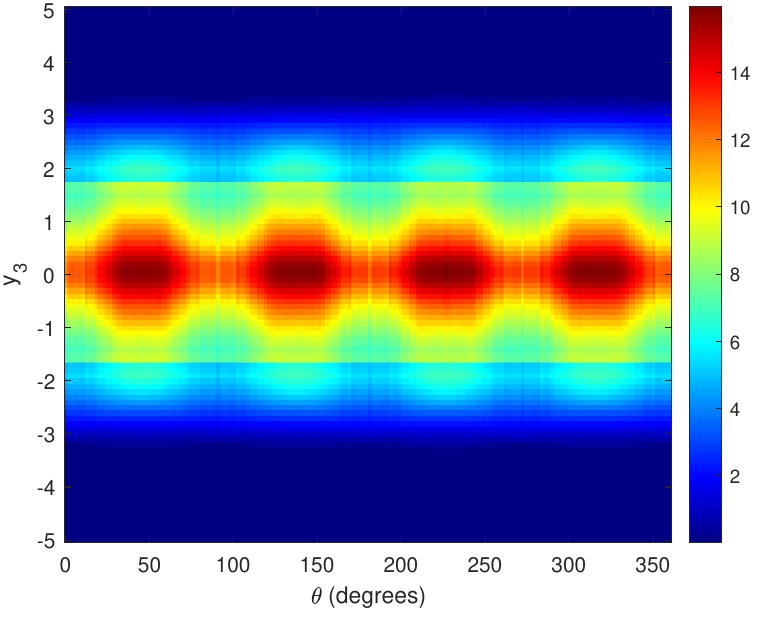}
\subcaption*{$\{s=1.2\}$ plane}
\end{subfigure}
\begin{subfigure}{0.24\textwidth}
\includegraphics[width=0.9\linewidth, height=3.2cm, keepaspectratio]{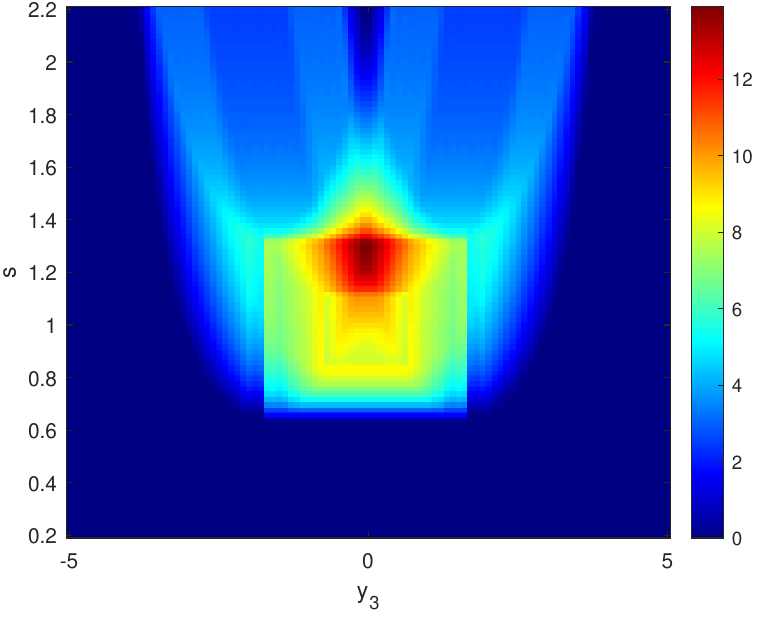}
\subcaption*{$\{\theta = 0\}$ plane}
\end{subfigure}
\begin{subfigure}{0.24\textwidth}
\includegraphics[width=0.9\linewidth, height=3.2cm, keepaspectratio]{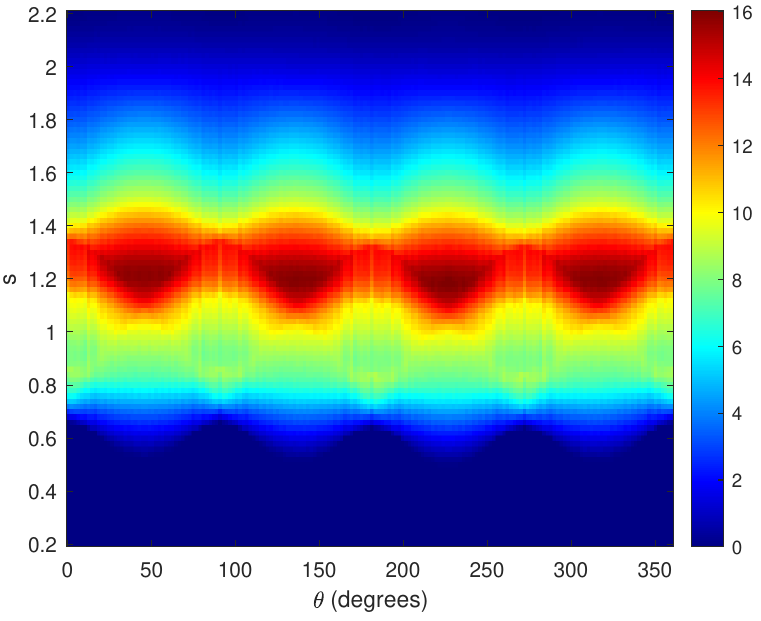}
\subcaption*{$\{y_3 = 0\}$ plane}
\end{subfigure}
\caption{Example $R_{\mu}f$ sinograms when $f$ is a hollow cuboid. Sinograms are generated for three integral surfaces, namely spheres (top row), spheroids (middle row), and lemons (bottom row). We show three cross-sections for each integrals surface, which are specified in the sub-figure caption.}
\label{F_sino}
\end{figure}
The edges of the hollow cuboid phantom can be seen in the $\{\theta = 0\}$ plane cross-sections, although they are smoothed out due to the application of $R_{\mu}$. In the $\{y_3 = 0\}$ plane and $\{s=1.2\}$ plane cross-sections, the sinograms are $\pi/2$  periodic as a function of $\theta$. This is due to the four-fold rotational symmetry of the hollow cuboid about the $x_3$ axis.

\begin{figure}
\centering
\begin{subfigure}{0.24\textwidth}
\includegraphics[width=0.9\linewidth, height=3.2cm, keepaspectratio]{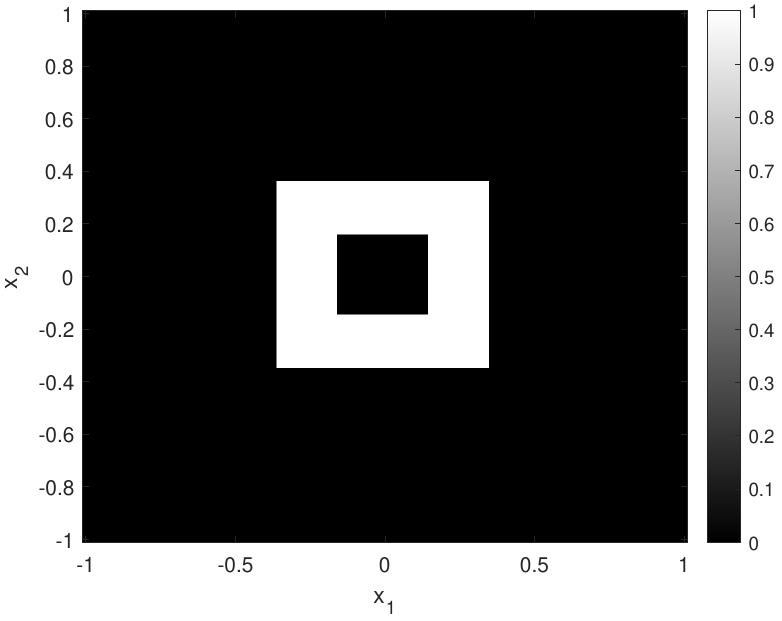}
\end{subfigure}
\begin{subfigure}{0.24\textwidth}
\includegraphics[width=0.9\linewidth, height=3.2cm, keepaspectratio]{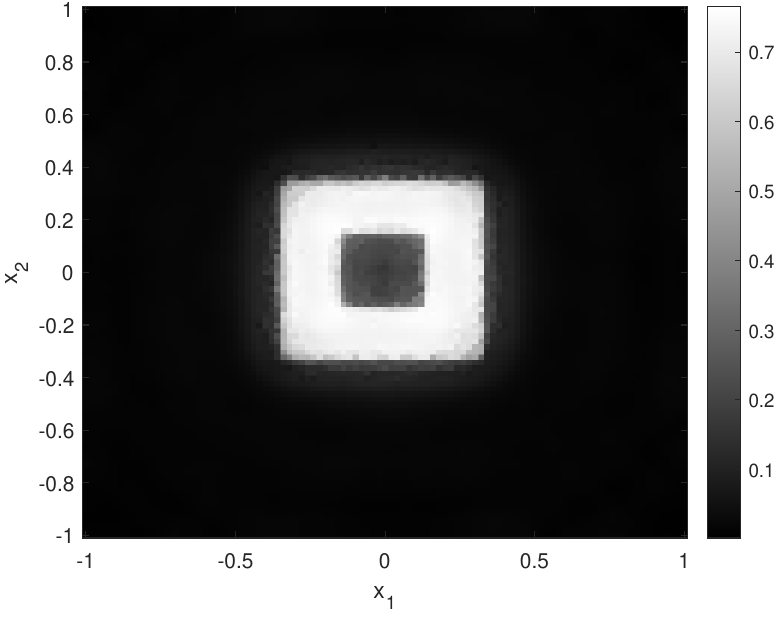}
\end{subfigure}
\begin{subfigure}{0.24\textwidth}
\includegraphics[width=0.9\linewidth, height=3.2cm, keepaspectratio]{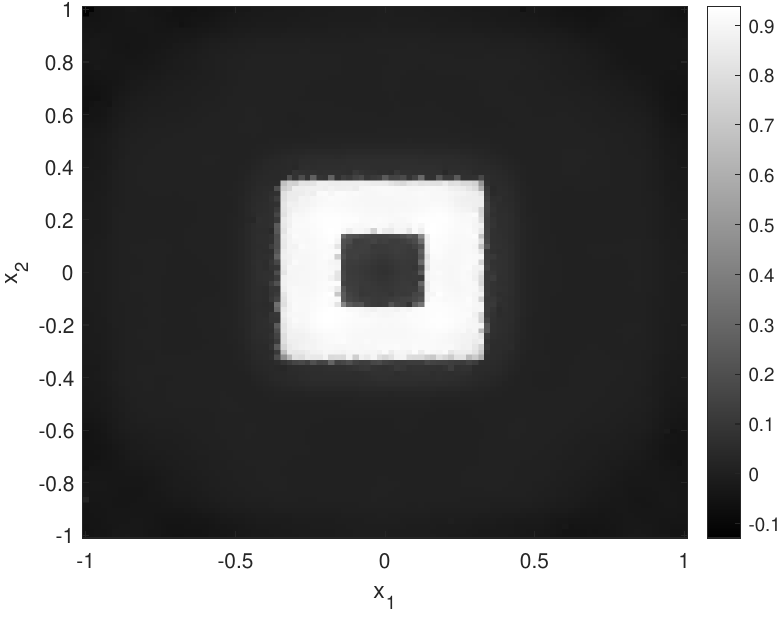}
\end{subfigure}
\begin{subfigure}{0.24\textwidth}
\includegraphics[width=0.9\linewidth, height=3.2cm, keepaspectratio]{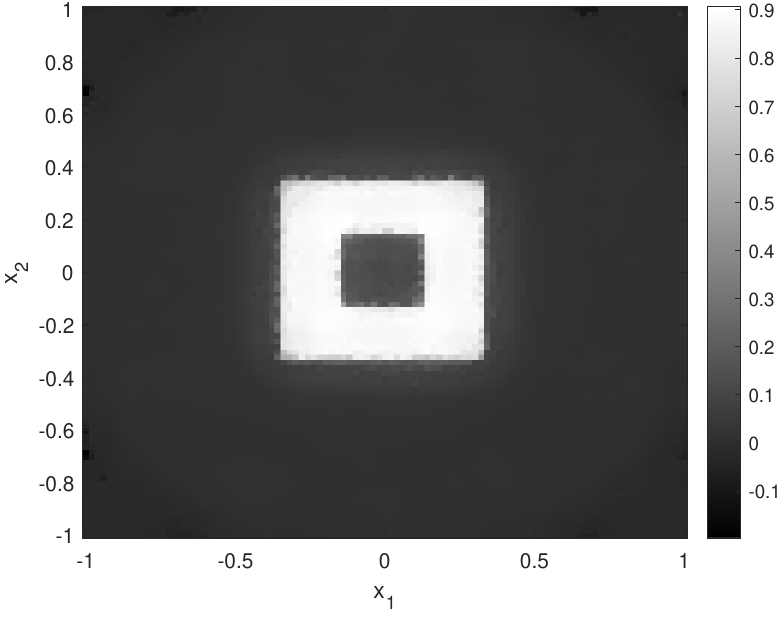}
\end{subfigure}
\begin{subfigure}{0.24\textwidth}
\includegraphics[width=0.9\linewidth, height=3.2cm, keepaspectratio]{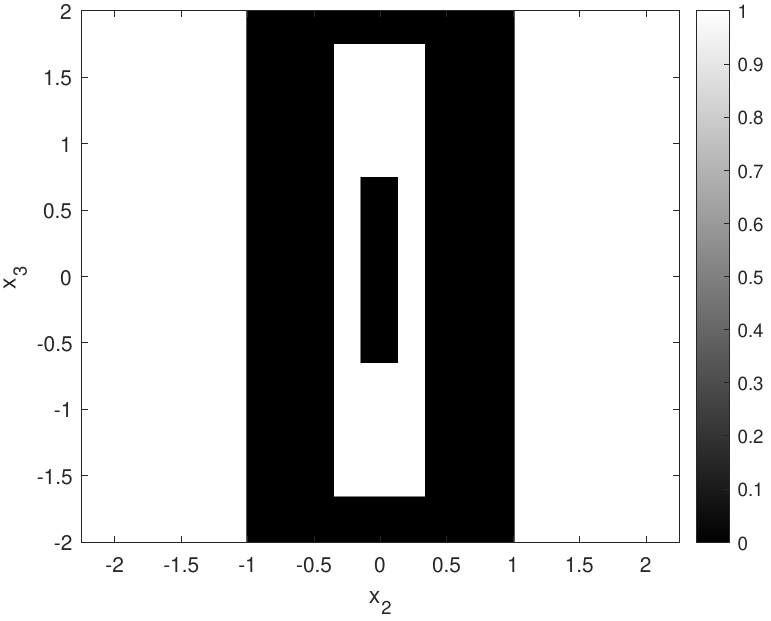}
%\subcaption{GT}
\end{subfigure}
\begin{subfigure}{0.24\textwidth}
\includegraphics[width=0.9\linewidth, height=3.2cm, keepaspectratio]{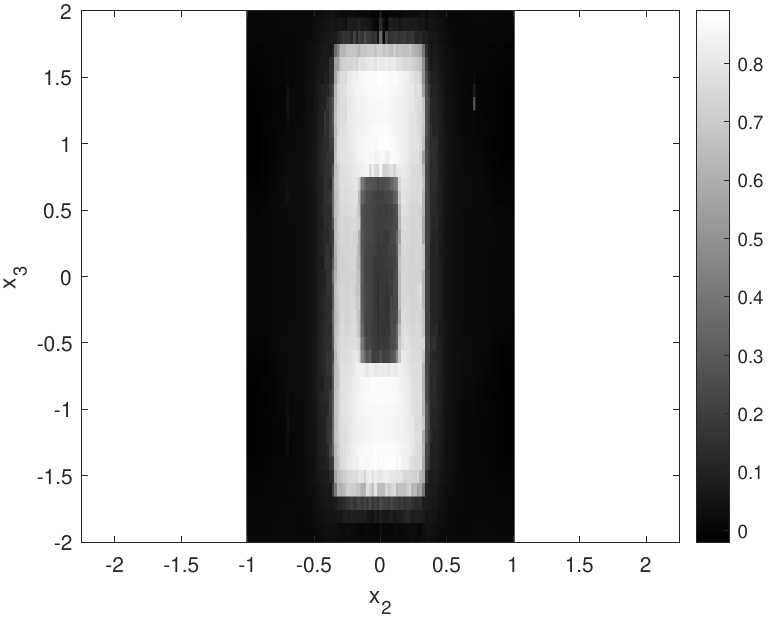}
%\subcaption{sphere}
\end{subfigure}
\begin{subfigure}{0.24\textwidth}
\includegraphics[width=0.9\linewidth, height=3.2cm, keepaspectratio]{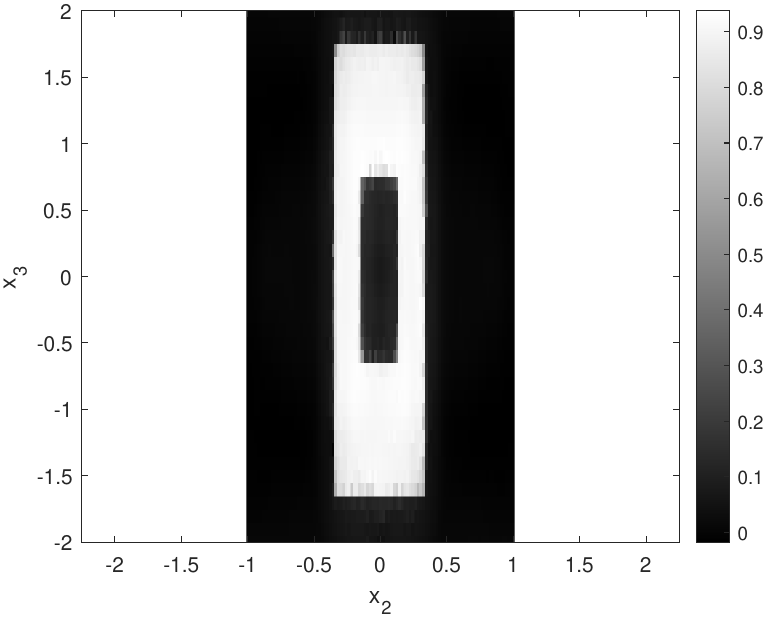}
%\subcaption{spheroid}
\end{subfigure}
\begin{subfigure}{0.24\textwidth}
\includegraphics[width=0.9\linewidth, height=3.2cm, keepaspectratio]{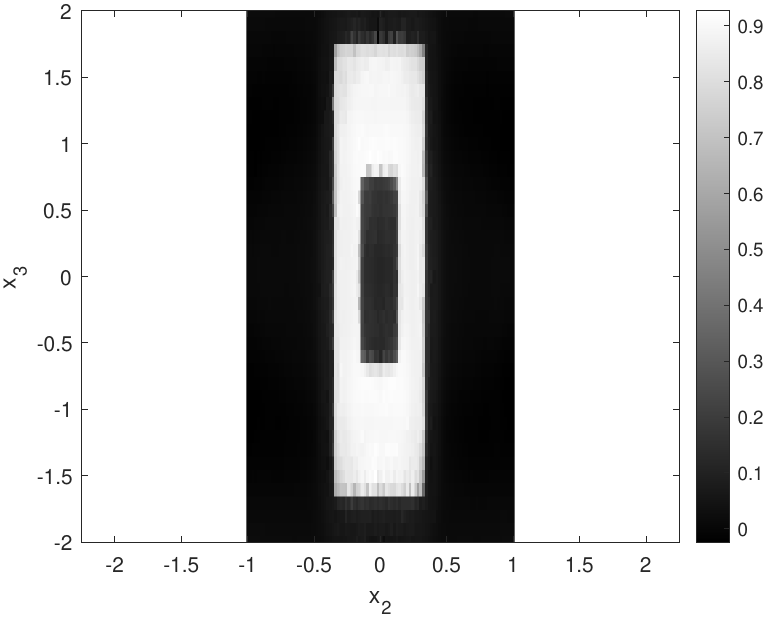}
%\subcaption{lemon}
\end{subfigure}
\begin{subfigure}{0.24\textwidth}
\includegraphics[width=0.9\linewidth, height=3.2cm, keepaspectratio]{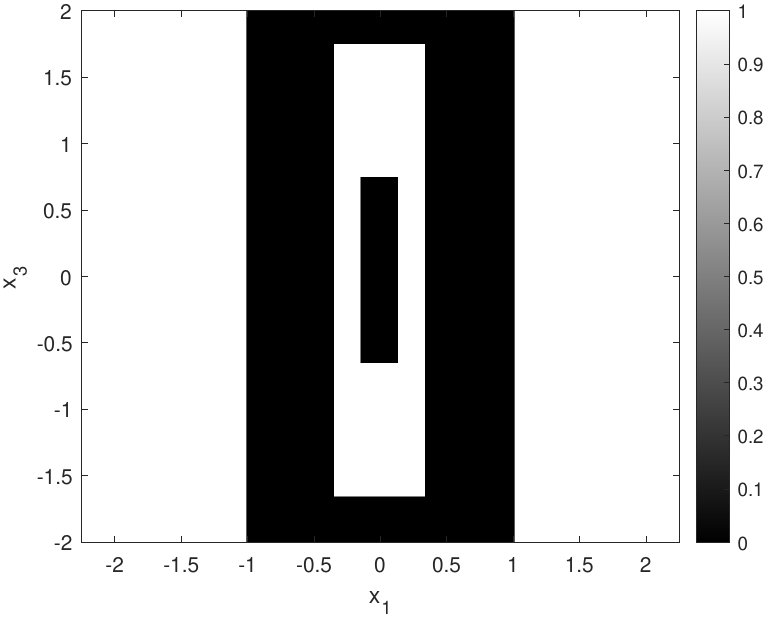}
\subcaption*{ground truth}
\end{subfigure}
\begin{subfigure}{0.24\textwidth}
\includegraphics[width=0.9\linewidth, height=3.2cm, keepaspectratio]{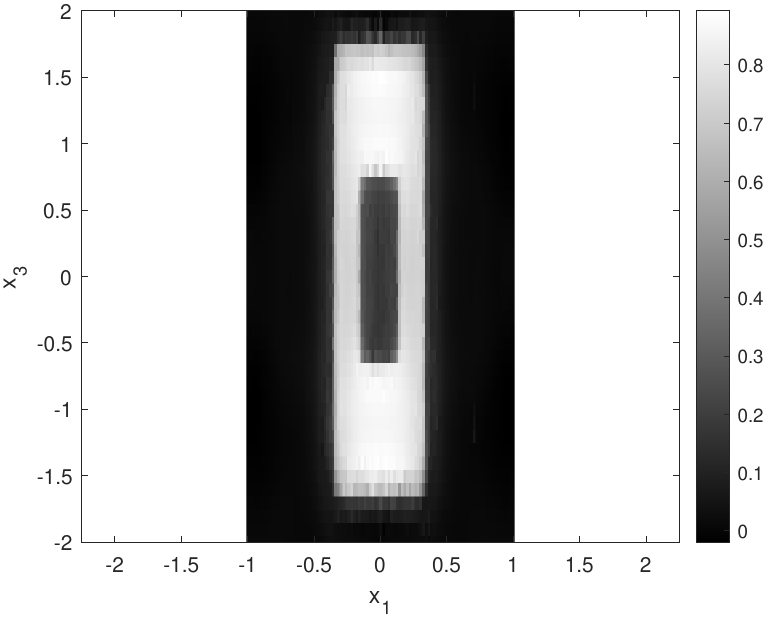}
\subcaption*{spheres}
\end{subfigure}
\begin{subfigure}{0.24\textwidth}
\includegraphics[width=0.9\linewidth, height=3.2cm, keepaspectratio]{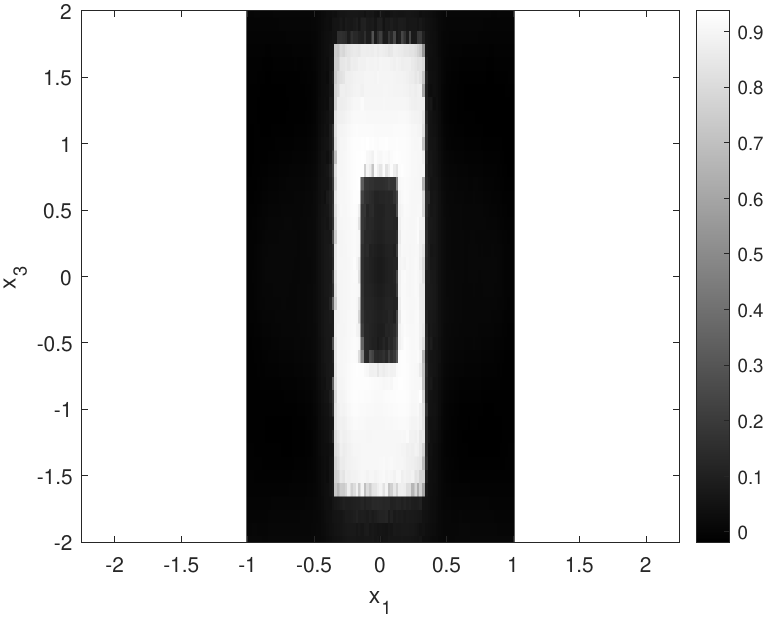}
\subcaption*{spheroids}
\end{subfigure}
\begin{subfigure}{0.24\textwidth}
\includegraphics[width=0.9\linewidth, height=3.2cm, keepaspectratio]{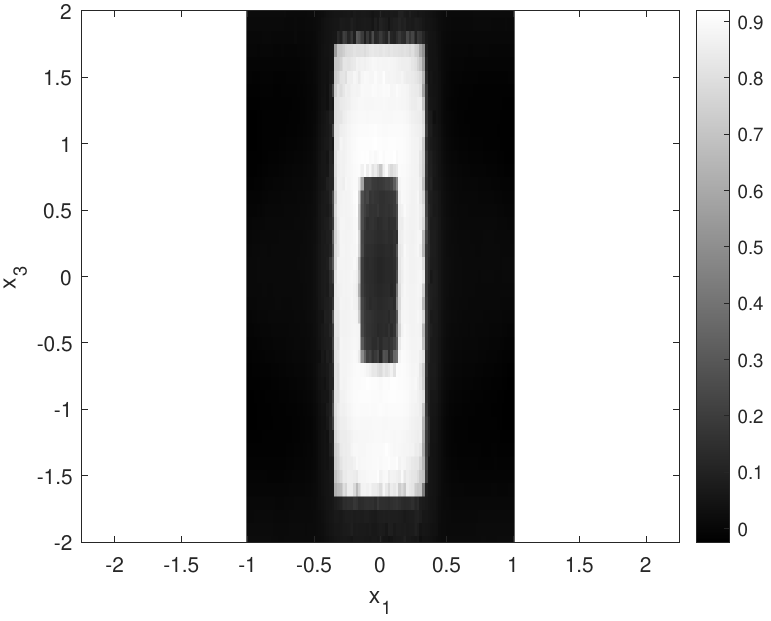}
\subcaption*{lemons}
\end{subfigure}
\caption{Reconstructions of hollow cuboid with $5\%$ added Gaussian noise. Top row - $(x_1,x_2)$ plane. Middle row - $(x_2,x_3)$ plane. Bottom row - $(x_1,x_3)$ plane.}
\label{F2}
\end{figure}

See figure \ref{F2}, where we have presented reconstructions of the hollow cube phantom from the sinogram data in figure \ref{F_sino}, with $5\%$ added Gaussian noise, and see figure \ref{errors} where we plot the relative least-squares reconstruction errors for varying levels of added Gaussian noise. The Tikhonov and TV smoothing parameters used for each noise level are given on the $x$ axis of figure \ref{errors}. To reconstruct the hollow cube phantom, we implemented the GCLS-TV algorithm as described in sub-section \ref{inv:methods}. The spheroid data reconstructions appear sharpest overall, when compared to the sphere and lemon reconstructions, and this is particularly noticeable in the $(x_2,x_3)$ and $(x_1,x_3)$ cross-sections. The sphere reconstructions are the most blurred, and offer the lowest image quality. This is verified by the reconstruction error plots in figure \ref{errors}, as the sphere integral reconstructions are shown to have the greatest error, when compared to lemon and spheroid integrals, and the spheroid reconstructions have the least error, particularly at higher noise levels (i.e., when the added noise is greater than or equal to $5\%$).
\begin{figure}[!h]
\centering
\begin{subfigure}{0.43\textwidth}
\includegraphics[width=1.2\linewidth, height=4.5cm, keepaspectratio]{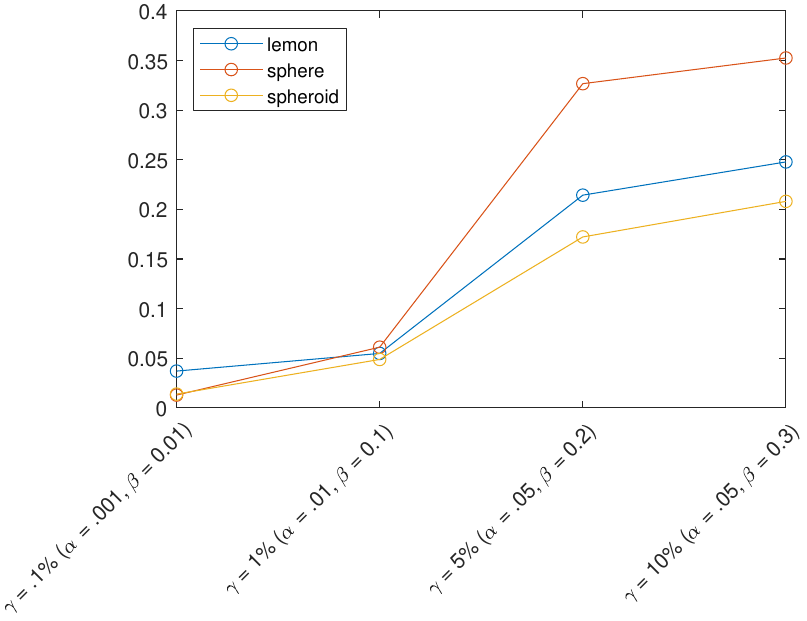} 
\subcaption{least-squares errors} \label{errors}
\end{subfigure}
\begin{subfigure}{0.43\textwidth}
\includegraphics[width=1.2\linewidth, height=4.5cm, keepaspectratio]{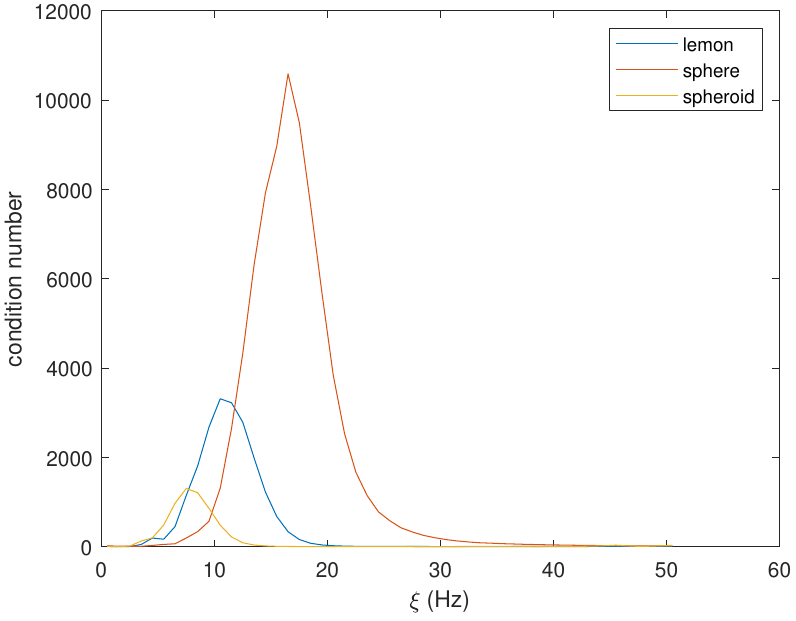}
\subcaption{condition numbers} \label{conds}
\end{subfigure}
\caption{(A) - least-squares error curves for varying noise levels ($\gamma$), corresponding to each integral surface considered. The Tikhonov and TV smoothing parameters, $\alpha$ and $\beta$, respectively, corresponding to each $\gamma$, are given on the $x$ axis in (A). (B) - condition number plots of the $V_{\xi}$ operators, for varying $\xi$ values, and for different integral surfaces of revolution.}
\label{F3}
\end{figure}

The sphere, spheroid, and lemon surfaces are defined using different $\mu_1$, which changes the $V_{\xi}$ operators in \eqref{sim}. The remaining operations in \eqref{sim}, e.g., $M$, remain constant, and do not vary with the integral surface. Thus, the differences in stability of inversion of $R_{\mu}$, for different integral surfaces, can be quantified through analysis of the $V_{\xi}$. To investigate this, we plot the condition numbers of the $V_{\xi}$ for varying $\xi$, and for sphere, spheroid, and lemon integral surfaces. See figure \ref{conds}. The largest peak in condition number occurs in the sphere curve, and the area under the sphere condition number curve is the largest. The spheroid condition numbers have the smallest peak, and area under the curve. Thus, the $V_{\xi}$ operators exhibit the greatest inversion instability in the sphere case, when compared to spheroids or lemons, which corresponds to greater noise amplification in the image reconstructions, and this is verified by figure \ref{errors}. More generally, we notice that the area under the condition number curve correlates positively with the area under the error curves in figure \ref{errors}, which is to be expected as the condition number bounds the least squares error. This suggests that, in certain geometries (such as considered here), reconstructing $f$ from spheroid data is preferred when compared to sphere data, in terms of inversion stability. This could have important implications in an application such as URT. In URT, the foci of the spheroid correspond to sound wave emitters and receivers. When the linear eccentricity, $c=0$, the spheroids reduce to spheres, and the sound waves are emitted and received at the same point. In this example, setting $c=2$ has advantages over the $c=0$ case in terms of inversion stability. This has implications regarding scanner design in URT, e.g., we could compare the area under the condition number curves for a range of $c$ to determine the optimal $c$, which, in the context of URT, is the distance between the emitter and receiver. This is an idea which warrants further investigation, which we aim to address in future work.

\section{Conclusion}
In this paper, we presented microlocal and injectivity analyses of a novel Radon transform, $R$, which defines the integrals of a function, $f$, over surfaces of revolution with centers on generalized surfaces $S = Q\times \mathbb{R}$ in $\mathbb{R}^n$, where $Q \subset \mathbb{R}^{n-2}$ is a smooth embedded hypersurface. In Theorem \ref{thm:bolker}, we analyzed $R$ as an FIO, and provided conditions on $h$ and $S$ which are necessary and sufficient for the Bolker condition to hold. The conditions on $h$ involve the first and second order derivatives of $h$, which can be calculated and verified simply for many examples of interest, e.g., in URT, when $h$ defines a semicircle or elliptic arc. The conditions on $S$ are geometric and require checking if planes tangent to $S$ intersect the support of $f$. In section \ref{inversion}, the surface of revolution transforms were shown to be closely related to the spherical Radon transform, and, using this idea, we proved injectivity results in Theorem \ref{thm:inversion}. 

In section \ref{images}, we presented condition number plots of the Volterra operators, $V_{\xi}$, which were used in the proof of Theorem \ref{thm:inversion}. The area under the condition number curve was shown to correlate positively with the least-squares reconstruction error. Interestingly, the condition number plots had an approximate bell shape, and peaked at intermediate frequency values, $\xi$. In further work, we aim to investigate in more detail how the location and size of the condition number peak relates to the Volterra operators in \eqref{volt_0}, and if we can predict the peak location and magnitude analytically.

\section*{Acknowledgements:} 
%We would like to thank Marika Horsky for translating the abstract from English into French.   
The first author wishes
to acknowledge funding support from Brigham Ovarian Cancer Research Fund, The V Foundation, Abcam Inc., and Aspira Women's Health.
Sean Holman was supported by the Engineering and Physical Sciences Research Council (EPSRC) grant number EP/V007742/1. The third author's research was partially supported
by Simons grant 708556. The authors would like to thank the Isaac Newton Institute for Mathematical Sciences, Cambridge, for support and hospitality during the programme Rich and Nonlinear Tomography where work on this paper was undertaken. This programme was supported by EPSRC grant number EP/R014604/1.

\bibliographystyle{abbrv} 
\bibliography{RefRevolution}

\end{document}